\documentclass[11pt]{article}
\usepackage[a4paper,left=2cm,right=2cm,top=1.5cm,bottom=2cm]{geometry}
\usepackage{hyperref}
\usepackage[english]{babel}
\usepackage{amsmath,amssymb,graphicx,enumerate,latexsym,xcolor,theorem}

\catcode`\<=\active \def<{
\fontencoding{T1}\selectfont\symbol{60}\fontencoding{\encodingdefault}}
\catcode`\>=\active \def>{
\fontencoding{T1}\selectfont\symbol{62}\fontencoding{\encodingdefault}}
\newcommand{\assign}{:=}
\newcommand{\backassign}{=:}
\newcommand{\cdummy}{\cdot}
\newcommand{\nin}{\not\in}
\newcommand{\nosymbol}{}

\newcommand{\tmmathbf}[1]{\ensuremath{\boldsymbol{#1}}}
\newcommand{\tmop}[1]{\ensuremath{\operatorname{#1}}}
\newcommand{\tmstrong}[1]{\textbf{#1}}
\newcommand{\tmtextit}[1]{\text{{\itshape{#1}}}}
\newenvironment{enumerateroman}{\begin{enumerate}[i.] }{\end{enumerate}}
\newenvironment{proof}{\noindent\textbf{Proof\ }}{\hspace*{\fill}$\Box$\medskip}
\newtheorem{definition}{Definition}
\newtheorem{lemma}{Lemma}
\newtheorem{notation}{Notation}
\newtheorem{proposition}{Proposition}
{\theorembodyfont{\rmfamily}\newtheorem{remark}{Remark}}
\newtheorem{theorem}{Theorem}

\newcommand{\zzone}{\text{\resizebox{.7em}{!}{\includegraphics{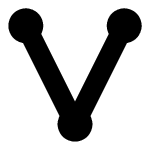}}}}
\newcommand{\zztwo}{\text{\resizebox{.7em}{!}{\includegraphics{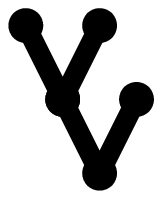}}}}

\numberwithin{equation}{section}
\begin{document}

\title{Finite speed of propagation for the 2- and 3-dimensional multiplicative stochastic wave equation}

\author{
  Immanuel Zachhuber}
\date{\today}

\maketitle

 \begin{abstract} \noindent
 	We prove finite speed of propagation for the multiplicative stochastic
	wave equation in two and three dimensions which leads us to a global
	space-time well-posedness result for the cubic nonlinear equation in the analogue of the energy space.
\end{abstract}
\section{Introduction}

The aim of this paper is to solve the \tmtextit{cubic multiplicative stochastic wave
equation} globally in space and time, which is formally
\begin{eqnarray}
  \partial^2_t u - \Delta u - u \cdummy \xi  &=& - u^3  \text{ on }
  \mathbb{R}_+ \times \mathbb{R}^d  \label{eqn:Wavenaive}\\
  (u, \partial_t u) |_{t = 0} &=& (u_0, u_1) \nonumber
\end{eqnarray}
for $d = 2, 3$ and $\xi$ being spatial white noise, see Section
\ref{sec:noise} for a rigorous definition. This continues the investigation of
this equation from {\cite{GUZ}} where the equation was shown to be globally
well-posed in the periodic setting with data in the energy space. Moreover,
Strichartz estimates were shown to hold in 2 dimensions in
{\cite{mouzardstrichartz}} (following the work
{\cite{zachhuber2019strichartz}} on Strichartz estimates on the Schr{\"o}diger
analogue of the equation considered here). To be precise, the result from
{\cite{GUZ}} is that until any time $T > 0$ one has a unique solution to
\begin{eqnarray}
  \partial^2_t u - H u & = & - u^3  \text{ on } [0, T] \times \mathbb{T}^d 
  \label{eqn:Wavetorus}\\
  (u, \partial_t u) |_{t = 0} & = & (u_0, u_1) \in \mathcal{D} \left( \sqrt{-
  H} \right) \times L^2, \nonumber
\end{eqnarray}
with continuous dependence on the data. Here $H$ denotes the continuous
Anderson Hamiltonian on $\mathbb{T}^d = (\mathbb{R}/\mathbb{Z})^d$, $d = 2,
3$, which is formally
\begin{equation}
  H = \Delta + \xi - \infty \label{Hformal}
\end{equation}
and $\mathcal{D} \left( \sqrt{- H} \right)$ denotes its \tmtextit{form
domain}, i.e. the functions $u \in L^2 (\mathbb{T}^d)$ s.t.
\[ | (u, H u)_{L^2} | < \infty . \]
The operator $H$ which is self-adjoint and semibounded from above was first
constructed on the \ 2 dimensional torus in {\cite{allez_continuous_2015}} and
later in three dimensions in {\cite{GUZ}}, both using the theory of
\tmtextit{Paracontrolled Distributions} which was introduced in {\cite{GIP}}.
There is also a construction using \tmtextit{Regularity Structures} (the
theory due to Hairer {\cite{hairer_theory_2014}}) by Labb{\'e} {\cite{labbe}}.
Since there are additional difficulties that arise in this approach on the
whole space, we do not use it meaning, in particular, that we do not need
either theory of singular SPDE but an exponential transform as in
{\cite{HairerLabbe15,debussche2016schr}} and {\cite{GUZ}} which is enough to
remove the most singular part of the noise.
See also the recent preprint \cite{jape} where a similar trick, however with a space-time noise, was applied to the dynamical $\Phi_4^3$ equation.
The main difficulty to get global space-time solutions as opposed to the
equation \eqref{eqn:Wavetorus} in the periodic setting is that the noise term
appearing is not only \tmtextit{rough} but also \tmtextit{unbounded}. For the
PDE this leads to the conserved energy no longer controlling the relevant
norms which is related to the fact that the Anderson Hamiltonian $H$ on the
whole space will no longer be semibounded as it is on the torus {\cite{GUZ}},
see {\cite{CvZ}} The remedy is to prove that one has \tmtextit{finite speed of
propagation} for this kind of equation, meaning that the solution at a given
space-time point will only depend on the ``backward light cone'' which is
well-known for the classical wave equation, see e.g. {\cite{Evans10}}. This
means that it is enough in some sense to solve ``locally'' which means that
the unboundedness of the noise does not play as much of a role. Due to the
presence of irregular objects, we use the approach due to Tartar
{\cite{tartar}} which was applied to somewhat similar situations in the more
recent works {\cite{BRfinite}} and {\cite{LHfinite}}.

\

Let us mention here also a couple of somewhat related recent papers: In
{\cite{tolomeoglobal}} global space-time solutions to the 2 dimensional
\tmtextit{cubic addititive stochastic nonlinear wave equation}(first solved in
the periodic setting in {\cite{GKO2}}) were constructed using the finite speed
of propagation of the (classical) wave equation together with an argument
based on the I-method; in {\cite{debglobal}} the two dimensional
Schr{\"o}dinger analogue of the equation considered here (first solved on the
torus in {\cite{debussche2016schr}}) was solved on the full space with some
range of power nonlinearities.

\

\

We introduce some notation and conventions which will be used frequently
throughout the paper. For the majority of the article we will be on the three
dimensional euclidean space $\mathbb{R}^3$ ( noting that the two dimensional
case is analogous but simpler) meaning that we often omit it
from our notation, meaning we write e.g.
\[ L^p = L^p (\mathbb{R}^d) = \left\{ f : \mathbb{R}^d \rightarrow
   \mathbb{R}s.t. \| f \|_{L^p} \assign \left( \int_{\mathbb{R}^d} | f |^p
   \right)^{\frac{1}{p}} < \infty \right\} \]
for the Lebesgue spaces (the case $p = \infty$ having the usual modification),
$\mathcal{H}^s =\mathcal{H}^s (\mathbb{R}^d)$ for the \tmtextit{Sobolev
spaces} and $B_{p, q}^s = B_{p, q}^s (\mathbb{R}^d)$ for the Besov spaces, see
the appendix for the definition of these spaces. Whenever this is not the
case, e.g. if we state a result valid on the torus, we demark this explicitly.
For the ball of radius $r > 0$ around $x \in \mathbb{R}^3$ we write $B (x, r)
\assign \{ y \in \mathbb{R}^3 : | y - x | \leqslant r \} $ and $B (r) = B (0,
r) .$ In general we use the convention that generic constants may change from
line to line, we also write
\[ \lesssim \text{ } \Leftrightarrow \text{ } \leqslant \text{ up to a
   generic constant and } C (X) = C_X  \text{for a constant depending on the
   quantity } X \]
and we also frequently write things like
\[ \| f \|_{\mathcal{H}^{s - \varepsilon}} \lesssim 1 \text{ for any }
   \varepsilon > 0 \text{to mean that } \| f \|_{\mathcal{H}^{s -
   \varepsilon}} \leqslant C (\varepsilon)  \text{with } C (\varepsilon)
   \uparrow \infty \text{as } \varepsilon \rightarrow 0. \]
We also frequently write $(\cdummy, \cdummy)$ to mean a dual pairing without
having to specify exactly in which spaces, for example for $f \in
\mathcal{H}^{- \varepsilon}$ for $\varepsilon > 0$ and $g \in \mathcal{H}^1$
we would write
\[ (f, g) = (f, g)_{\mathcal{H}^{- \varepsilon}, \mathcal{H}^{\varepsilon}} \]
etc. In some cases it will be important exactly which pairing it is and then
we will use the latter notation. As is customary, we also write $\chi_A$ as
the indicator function of the set $A \subset \mathbb{R}^d$ i.e.
\[ \chi_A (y) = \left\{\begin{array}{ll}
     1 & \text{ if } y \in A\\
     0 & \text{ if } y \nin A
   \end{array}\right. . \]

The paper is organised as follows:

Section \ref{sec:exp} details how one transforms the equation using the
exponential of a stochastic object as was done in {\cite{GUZ}} in the periodic
setting and how it has to be modified on the whole space. In Section
\ref{sec:locstoc} we recall the localising operators from {\cite{GHglobal}}
and the existence and convergence properties of the stochastic objects
appearing.

In Section \ref{sec:trunc} we solve a suitably renormalised and truncated
version of \eqref{eqn:Wavenaive} globally in space-time, which is analogous to
how the periodic version \eqref{eqn:Wavetorus} was solved globally in time in
{\cite{GUZ}}.

Thereafter, in Section \ref{sec:speed}, we prove finite speed of propagation
for the linear multiplicative stochastic wave equation which is the main tool
of how to globalise the solutions we constructed in Section \ref{sec:trunc}.
Our method is based on an approach due to Tartar {\cite{tartar}}, which we
will also recall.

Lastly, in Section \ref{sec:cubicfinite}, we apply the results from Section
\ref{sec:speed} to the nonlinear equation \eqref{eqn:Wavenaive} in order to
get finite speed of propagation and consequently a global-in-time
well-posedness result.
\subsection*{Acknowledgments}
 The author would like to thank Antoine Mouzard for some helpful comments and Massimiliano Gubinelli for pointing out an alternative approach sketched in Remark \ref{rem:max}. 
\section{The exponential transformation and the noise terms}\label{sec:noise}

In {\cite{GUZ}} the authors used an exponential transform inspired by
{\cite{debussche2016schr}}--which in turn was inspired by
{\cite{HairerLabbe15}}-- in order to remove the worst part of the irregularity
of the noise. It turns out that if one is content with constructing a
form-domain(as we are here) this is sufficient, see Section 2.2 in
{\cite{GUZ}}. In Section \ref{sec:exp} we first recall this construction on
the torus and how to extend it to the whole space, or at least a localised
version of it. Then in Section \ref{sec:locstoc} we recall some results about
localising operators from {\cite{GHglobal}} and the existence and regularity
properties of the relevant stochastic objects.

\subsection{The exponential transform on the torus vs. the whole
space}\label{sec:exp}

We make a similar computation as in Section 2.2 of {\cite{GUZ}} in 3
dimensions with the 2 dimensional case being a bit simpler. Initially we
recall the computation which works on the torus and then detail how one has to
modify it in order to make it work on the whole space.

We start, using the same notation as in {\cite{GUZ}}, with $\xi \in
\mathcal{C}^{- \frac{3}{2} - \varepsilon} (\mathbb{T}^3), \varepsilon > 0$
chosen very small, a \tmtextit{spatial white noise} on $\mathbb{T}^3$, see the
appendix for the definition if the H{\"o}lder-Besov spaces. Then we formally
set (we abuse notation a bit here since we will have to slightly redefine some
of the objects later when working on the whole space)
\begin{align}
  X \assign & (1 - \Delta)^{- 1} \xi (x) && \in \mathcal{C}^{\frac{1}{2} -
  \varepsilon} (\mathbb{T}^3) \\
  X^{\zzone} \assign & (1 - \Delta)^{- 1} | \nabla X |^2 && \in \mathcal{C}^{1
  - \varepsilon} (\mathbb{T}^3)  \label{eqn:X2formal}\\
  X^{\zztwo} \assign & 2 (1 - \Delta)^{- 1} \left( \nabla X \cdot \nabla
  X^{\zzone} \right) && \in \mathcal{C}^{\frac{3}{2} - \varepsilon}
  (\mathbb{T}^3) \\
  W \assign & X + X^{\zzone} + X^{\zztwo} && \in \mathcal{C}^{\frac{1}{2} -
  \varepsilon} (\mathbb{T}^3) 
\end{align}
and make the following ansatz for the form domain of $\Delta + \xi$
\begin{equation}
  u = e^W v = e^{X + X^{\zzone} + X^{\zztwo}} v, \label{ansatz:exp1}
\end{equation}
where the regularity of $v$ will be specified later. We begin by computing
formally
\begin{align*}
  \Delta u + u \xi  = & e^W \left( \Delta v + \Delta W v + \left| \nabla
  \left( X + X^{\zzone} + X^{\zztwo} \right)  \right|^2 v + 2 \nabla W \cdummy
  \nabla v + v \xi \right)\\
   = & e^W \Big( \Delta v + \Bigg( \overset{\text{not
  def.}}{\overbrace{\left| \nabla X^{\zzone} \right|^2}} + \left| \nabla
  X^{\zztwo} \right|^2 + 2 \overset{\text{not def.}}{\overbrace{\nabla X
  \cdummy \nabla X^{\zztwo}}} + 2 \nabla X^{\zzone} \cdummy \nabla X^{\zztwo}
  - X - X^{\zzone} - X^{\zztwo} \Bigg) v + \\  
  & +2 \nabla W \cdummy \nabla v
  \Big),
\end{align*}
so we see some cancellations happening from our choice of $W$. However, the
right-hand side contains some terms which are undefined (in fact, even the
term $| \nabla X |^2$ appearing in \eqref{eqn:X2formal} is not defined),
ultimately one is able to probabilistically give a rigorous meaning to
\begin{eqnarray*}
  : | \nabla X |^2 :  \text{``$=$''}& \left| \nabla X \right|^2 - \infty & \in
  \mathcal{C}^{- 1 - \varepsilon} (\mathbb{T}^3)\\
  : \left| \nabla X^{\zzone} \right|^2 :   \text{``$=$''}& \left| \nabla
  X^{\zzone} \right|^2 - \infty & \in \mathcal{C}^{- \varepsilon}
  (\mathbb{T}^3)\\
  \text{and } & \nabla X \cdummy \nabla X^{\zztwo} & \in \mathcal{C}^{-
  \frac{1}{2} - \varepsilon} (\mathbb{T}^3),
\end{eqnarray*}
for $\varepsilon > 0$, where the colons denote \tmtextit{Wick ordering} which
is a kind of \tmtextit{renormalisation}, see Theorem \ref{thm:noise} for a
rigorous statement. This is the origin of the formal ``$- \infty$'' appearing
in \eqref{Hformal}.

In this way, we can define the operator $\bar{H}$ as the ``correct''
renormalised version of $\Delta + \xi$
\[ \bar{H} (e^W v) \assign e^W \left( \Delta v + \overset{\backassign \bar{Z}
   \in \mathcal{C}^{- \frac{1}{2} - \varepsilon}
   (\mathbb{T}^3)}{\overbrace{\left( : \left| \nabla X^{\zzone} \right|^2 : +
   \left| \nabla X^{\zztwo} \right|^2 + 2 \nabla X \cdummy \nabla X^{\zztwo} +
   2 \nabla X^{\zzone} \cdummy \nabla X^{\zztwo} - W \right)}} v + 2 \nabla W
   \cdummy \nabla v \right), \]
for $W \assign X + X^{\zzone} + X^{\zztwo}$ and the ``corrected'' $X^{\zzone}
\assign (1 - \Delta)^{- 1} : | \nabla X |^2 :$.

Now the observation in {\cite{GUZ}} is that for $v \in \mathcal{H}^1
(\mathbb{T}^3),$ one actually has that
\begin{equation}
  | (e^W v, \bar{H} (e^W v)) | < \infty \Leftrightarrow v \in \mathcal{H}^1
  (\mathbb{T}^3), \label{eqn:Hbarform}
\end{equation}
meaning that $e^W \mathcal{H}^1 (\mathbb{T}^3)$ is the \tmtextit{form domain
}of the operator $\bar{H}$. In fact, by integrating by parts one sees
\begin{eqnarray*}
  (e^W v, \bar{H} (e^W v)) & = & (e^{2 W} v, \Delta v + 2 \nabla W \cdummy
  \nabla v + \bar{Z} v)\\
  & = & (e^{2 W} v, 2 \nabla W \cdummy \nabla v + \bar{Z} v) - (e^{2 W}
  \nabla v, \nabla v) - (e^{2 W} v 2 \nabla W, \nabla v)\\
  & = & (e^{2 W} v, \bar{Z} v) - (e^{2 W} \nabla v, \nabla v)
\end{eqnarray*}
and, since one can show $e^{2 W} \bar{Z} \in \mathcal{C}^{- \frac{1}{2} -
\varepsilon} (\mathbb{T}^3)$ (see Lemma 2.40 in {\cite{GUZ}}) one has that
($C_{\Xi}$ denotes a changing constant which depends only on the H{\"o}lder
norms of the noise objects)
\begin{eqnarray*}
  | (e^W v, \bar{H} (e^W v)) | & \leqslant & \| e^{2 W} \|_{L^{\infty}
  (\mathbb{T}^3)} (\nabla v, \nabla v)_{L^2 (\mathbb{T}^3)} + \| e^{2 W}
  \bar{Z} \|_{\mathcal{C}^{- \frac{1}{2} - \varepsilon} (\mathbb{T}^3)} \| v^2
  \|_{B^{\frac{1}{2} + \varepsilon}_{1, 1} (\mathbb{T}^3)}\\
  & \leqslant & \| e^{2 W} \|_{L^{\infty} (\mathbb{T}^3)} (\nabla v, \nabla
  v)_{L^2 (\mathbb{T}^3)} + \| e^{2 W} \bar{Z} \|_{\mathcal{C}^{- \frac{1}{2}
  - \varepsilon} (\mathbb{T}^3)} \| v \|_{B^{\frac{1}{2} + \varepsilon}_{2, 1}
  (\mathbb{T}^3)} \| v \|_{L^2 (\mathbb{T}^3)}\\
  & \leqslant & C_{\Xi} ((\nabla v, \nabla v)_{L^2 (\mathbb{T}^3)} + (v,
  v)_{L^2 (\mathbb{T}^3)})
\end{eqnarray*}
having used Besov duality, Lemma \ref{lem:leib}, and interpolation/Young's
inequality as well as the fact that both $e^{\pm W}$ are bounded in
$L^{\infty}$. Analogously one can bound the $\mathcal{H}^1 (\mathbb{T}^3)$
norm of $v$ by using
 \begin{align*}
     (\nabla v, \nabla v)_{L^2 (\mathbb{T}^3)} + (v, v)_{L^2 (\mathbb{T}^3)}
     \leqslant & C_{\Xi} (e^W \nabla v, e^W \nabla v)_{L^2 (\mathbb{T}^3)} +
     (e^W v, e^W v)_{L^2 (\mathbb{T}^3)}\\
     = & - C_{\Xi} (e^W v, \bar{H} (e^W v)) - C_{\Xi} (e^{2 W} v, \bar{Z} v) +
     (e^W v, e^W v)_{L^2 (\mathbb{T}^3)}
   \end{align*} 
and proceeding as before to bound the term containing $\bar{Z},$ thus one has
shown \eqref{eqn:Hbarform} and even norm equivalence.

\

\

Now we want to adapt the same approach to the whole space $\mathbb{R}^3$. The
thing that makes the problem on the whole space more difficult is that all the
quantities are unbounded, i.e. live only in weighted H{\"o}lder spaces (see
Definition \ref{def:besov} for the definition of weighted Besov spaces of
which H{\"o}lder spaces are a particular case and Section \ref{sec:locstoc}
for a discussion of how to define the stochastic objects).

In particular it is not good to have unbounded terms inside the exponential,
since we later want to use that the exponential and its inverse are both
bounded in $L^{\infty}$. By making use of the localising operators from
{\cite{GHglobal}} whose definition and properties we recall in Section
\ref{sec:locstoc}, we modify \eqref{ansatz:exp1} in the following way
\begin{equation}
  u = e^{W_{>}} v \label{ansatz:exp2},
\end{equation}
where $W_{>}$ can be thought of capturing the high frequencies and analogously
to the periodic case the stochastic terms are
\begin{align}
  X \assign & (1 - \Delta)^{- 1} \xi (x) & \in \mathcal{C}_{\langle \cdummy
  \rangle^{- \delta}}^{\frac{1}{2} - \varepsilon} \\
  X^{\zzone} \assign & (1 - \Delta)^{- 1} : | \nabla X |^2 : & \in
  \mathcal{C}_{\langle \cdummy \rangle^{- \delta}}^{1 - \varepsilon} \\
  X^{\zztwo} \assign & 2 (1 - \Delta)^{- 1} \left( \nabla X \cdot \nabla
  X^{\zzone} \right) & \in \mathcal{C}_{\langle \cdummy \rangle^{-
  \delta}}^{\frac{3}{2} - \varepsilon} \\
  W \assign & X + X^{\zzone} + X^{\zztwo} & \in \mathcal{C}_{\langle \cdummy
  \rangle^{- \delta}}^{\frac{1}{2} - \varepsilon}  \label{def:W}\\
  = & \underset{\in \mathcal{C}^{\frac{1}{2} - 2
  \varepsilon}}{\underbrace{W_{>}}} + \underset{\in \mathcal{C}_{\langle
  \cdummy \rangle^{\sigma}}^{2 + \varepsilon}}{\underbrace{W_{\leqslant}}} 
  \hspace{1cm}\text{for suitable } \sigma > 0, \text{see Proposition \ref{prop:loc}}&
  \nonumber
\end{align}
and we redo the above formal computation instead with the ansatz
\eqref{ansatz:exp2}
\begin{align*}
  \Delta u + u \xi = & e^{{W_{>}} } (\Delta W_{>} v + | \nabla W_{>}  |^2 v
  + \Delta v + 2 \nabla W_{>} \cdummy \nabla v + v \xi)\\
  = & e^{{W_{>}} } ((\Delta W - \Delta W_{\leqslant}) v + | \nabla W -
  \nabla W_{\leqslant}  |^2 v + \Delta v + 2 \nabla W_{>} \cdummy \nabla v + v
  \xi)\\
  = & e^{{W_{>}} } \Big( \big( W - \Delta W_{\leqslant} - \xi - : |
  \nabla X |^2 : - 2 \nabla X \cdot \nabla X^{\zzone} + | \nabla W |^2 - 2
  \nabla W \cdummy \nabla W_{\leqslant} + | \nabla W_{\leqslant} |^2 + \xi
  \big) v \\& \ + \Delta v + 2 \nabla W_{>} \cdummy \nabla v \Big)\\
  = & e^{{W_{>}} } \Big( \big( W - \Delta W_{\leqslant} + (| \nabla X |^2
  - : | \nabla X |^2 :) + \left| \nabla X^{\zzone} \right|^2 + \left| \nabla
  X^{\zztwo} \right|^2 + 2 \nabla X \cdummy \nabla X^{\zztwo} + 2 \nabla
  X^{\zzone} \cdummy \nabla X^{\zztwo}- \\ &\ - 2 \nabla W \cdummy \nabla
  W_{\leqslant} + | \nabla W_{\leqslant} |^2 \big) v + \Delta v + 2 \nabla
  W_{>} \cdummy \nabla v \Big),
\end{align*}
which leads to the formal definition (the difference is that we replace the
ill-defined squares by their Wick ordered versions)
\begin{eqnarray}
  H (e^{W_{>}} v) & = & e^{{W_{>}} } (\Delta v + Z v + 2 \nabla W_{>} \cdummy
  \nabla v) \nonumber\\
  & = & e^{{W_{>}} } (\Delta v + Z_{>} v + Z_{\leqslant} v + 2 \nabla W_{>}
  \cdummy \nabla v)  \label{eqn:AHwhole}
\end{eqnarray}
having defined
\begin{equation}
  Z \assign  W - \Delta W_{\leqslant} + : \left| \nabla X^{\zzone}
  \right|^2 : + \left| \nabla X^{\zztwo} \right|^2 + 2 \nabla X \cdummy \nabla
  X^{\zztwo} + 2 \nabla X^{\zzone} \cdummy \nabla X^{\zztwo} - 2 \nabla W
  \cdummy \nabla W_{\leqslant} + | \nabla W_{\leqslant} |^2,
\end{equation}
which we have split as
\[ \mathcal{C}_{\langle \cdummy \rangle^{\gamma}}^{- \frac{1}{2} -
   \varepsilon} \ni Z = \underset{\in \mathcal{C}^{- \frac{1}{2} - 2
   \varepsilon}}{\underbrace{Z_{>}}} + \underset{\in L_{\langle \cdummy
   \rangle^{- \gamma'}}^{\infty}}{\underbrace{Z_{\leqslant}}}, \text{for }
   \gamma' > \gamma > 0 \text{ dictated by Proposition \ref{prop:loc}.} \]
We have thus split our formal operator $\Delta + \xi$ into a rough operator
whose treatment is analogous to the periodic case in {\cite{GUZ}} and an
unbounded but regular part. Since we will be interested in finite speed of
propagation which is a local concept, we will consider truncated operators
which are defined rigorously on functions of the form \eqref{ansatz:exp2} as
follows.

\begin{definition}
  \label{def:3ops}We define the following operators for smooth functions $v$
  and $R > 0$
  \begin{eqnarray}
    & H_R (e^{W_{>}} v) &\assign  e^{{W_{>}} } (\Delta v + Z_{>} v + \chi_{B
    (R)} Z_{\leqslant} v + 2 \nabla W_{>} \cdummy \nabla v) \\
    & H_{>} (e^{W_{>}} v) &\assign  e^{{W_{>}} } (\Delta v + Z_{>} v + 2
    \nabla W_{>} \cdummy \nabla v) \\
    && \text{and}  \nonumber\\
    & H_{\gg} (e^{W_{>}} v) &\assign  e^{{W_{>}} } (\Delta v + Z_{>} v -
    C_{\gg} (\Xi) v + 2 \nabla W_{>} \cdummy \nabla v) \\
    & &=  H_{>} (e^{W_{>}} v) - C_{\gg} (\Xi) e^{{W_{>}} } v \\
    & &=  H_R (e^{W_{>}} v) - \chi_{B (0, R)} Z_{\leqslant} v - C_{\gg} (\Xi)
    e^{{W_{>}} } v  \label{eqn:HRH}
  \end{eqnarray}
  where the constant $C_{\gg} (\Xi) > 0$ is chosen depending on the norms of
  the noise terms s.t.
  \begin{equation}
    - (e^{W_{>}} v, H_{\gg} (e^{W_{>}} v)) \geqslant \| e^{W_{>}} v
    \|^2_{L^2}, \label{eqn:Hpos}
  \end{equation}
  this is similar to Proposition 2.53 in {\cite{GUZ}}.
\end{definition}

\begin{lemma}
  \label{lem:3ops}The form domain of all three operators in Definition
  \ref{def:3ops} is $e^{W_{>}} \mathcal{H}^1$ and we have the bounds (note
  that because of \eqref{eqn:Hpos} $H_{\gg}$ is the only operator that has a
  definite sign)
  \begin{eqnarray*}
    | (e^{W_{>}} v, H_{>} (e^{W_{>}} v)) | & \leqslant & - (e^{W_{>}} v,
    H_{\gg} (e^{W_{>}} v))\\
    & \leqslant & - (e^{W_{>}} v, H_{>} (e^{W_{>}} v)) + C_{\gg} (\Xi) \|
    e^{W_{>}} v \|^2_{L^2}\\
     & \text{and} & \\
    | (e^{W_{>}} v, H_R (e^{W_{>}} v)) | & \leqslant & - 2 (e^{W_{>}} v,
    H_{\gg} (e^{W_{>}} v)) + C (R, \Xi) \| e^{W_{>}} v \|^2_{L^2}\\
    - (e^{W_{>}} v, H_{\gg} (e^{W_{>}} v)) & \leqslant & - 2 (e^{W_{>}} v, H_R
    (e^{W_{>}} v)) + C (R, \Xi) \| e^{W_{>}} v \|^2_{L^2},
  \end{eqnarray*}
  where $C_{\gg} (\Xi) > 0$ is the constant from Definition \ref{def:3ops} and
  $C (R, \Xi) > 0$ is a constant that may depend polynomially on $R.$
  
  \begin{proof}
    This follows from the definitions of the operators as well as the previous
    discussion on how to bound the $\bar{Z}$ term as well as the bound
    \[ \| \chi_{B (R)} Z_{\leqslant} v \|_{L^2} \leqslant \| v \|_{L^2} \|
       \chi_{B (R)} Z_{\leqslant} \|_{L^{\infty}} \lesssim \| v \|_{L^2}
       R^{\rho} \| Z_{\leqslant} \|_{L_{\langle \cdummy \rangle^{-
       \rho}}^{\infty}}, \text{ for } \rho > 0. \]
    which is where the polynomial dependence on $R$ comes from.
  \end{proof}
\end{lemma}

For the sake of completeness, we give the analogous statements in the
two-dimensional setting. It is enough to consider one renormalised quantity,
namely we have for the stochastic objects, $\xi^{(2)} \in \mathcal{C}_{\langle
\cdummy \rangle^{- \delta}}^{- 1 - \varepsilon} (\mathbb{R}^2)$ being the
spatial white noise on $\mathbb{R}^2,$
\begin{align*}
   X^{(2)} \assign (1 - \Delta)^{- 1} \xi^{(2)} & \in \mathcal{C}_{\langle
  \cdummy \rangle^{- \delta}}^{1 - \varepsilon} (\mathbb{R}^2) \\
   : | \nabla X^{(2)} |^2 : & \in \mathcal{C}_{\langle \cdummy \rangle^{-
  \delta}}^{- \varepsilon} (\mathbb{R}^2), 
\end{align*}
for $\varepsilon, \delta > 0$, which leads to the exponential ansatz $u =
e^{X^{(2)}_{>}} v$ with $v \in \mathcal{H}^1 (\mathbb{R}^2)$ for which one
formally has
\begin{eqnarray*}
  \Delta u + u \xi & = & e^{{X^{(2)}_{>}} } \left( \Delta X^{(2)}_{>} v +
  \left| {\nabla X^{(2)}_{>}}  \right|^2 v + \Delta v + 2 \nabla X^{(2)}_{>}
  \cdummy \nabla v + v \xi^{(2)} \right)\\
  & = & e^{{X^{(2)}_{>}} } ((\Delta X^{(2)} - \Delta X^{(2)}_{\leqslant}) v +
  | \nabla X^{(2)} - \nabla X^{(2)}_{\leqslant} |^2 v + \Delta v + 2 \nabla
  X^{(2)}_{>} \cdummy \nabla v + v \xi^{(2)})\\
  & = & e^{{X^{(2)}_{>}} } ((X^{(2)} - \Delta X^{(2)}_{\leqslant} + | \nabla
  X^{(2)} |^2 - 2 \nabla X^{(2)} \nabla X^{(2)}_{\leqslant} + | \nabla
  X^{(2)}_{\leqslant} |^2) v + \Delta v + 2 \nabla X^{(2)}_{>} \cdummy \nabla
  v)
\end{eqnarray*}
leading to the rigorous definition
\[ H^{(2)} (e^{X^{(2)}_{>}} v) \assign e^{X^{(2)}_{>}} (Z^{(2)} v + \Delta v
   + 2 \nabla X^{(2)}_{>} \cdummy \nabla v), \]
having defined
\[ Z^{(2)} \assign (X^{(2)} - \Delta X^{(2)}_{\leqslant} + | : \nabla X^{(2)}
   : |^2 - 2 \nabla X^{(2)} \nabla X^{(2)}_{\leqslant} + | \nabla
   X^{(2)}_{\leqslant} |^2) = \underset{\in \mathcal{C}^{- 2 \varepsilon}
   (\mathbb{R}^2)}{\underbrace{Z_{>}^{(2)}}} + \underset{\in L_{\langle
   \cdummy \rangle^{- \gamma}}^{\infty}
   (\mathbb{R}^2)}{\underbrace{Z_{\leqslant}^{(2)}}} \]
for some $\gamma > 0$ dictated by Proposition \ref{prop:loc}. This leads to
the truncated operator
\[ H^{(2)}_R (e^{X^{(2)}_{>}} v) \assign e^{X^{(2)}_{>}} (Z_{>}^{(2)} v +
   \chi_{B (R)} Z_{\leqslant}^{(2)} + \Delta v + 2 \nabla X^{(2)}_{>} \cdummy
   \nabla v) \]
and the uniformly positive operator
\begin{eqnarray}
  H^{(2)}_{\gg} (e^{X^{(2)}_{>}} v) \assign & e^{X^{(2)}_{>}} (\Delta v +
  X^{(2)}_{>} v - C_{\gg} (\Xi^{(2)}) v + 2 \nabla X^{(2)}_{>} \cdummy \nabla
  v), 
\end{eqnarray}
where $C_{\gg} (\Xi^{(2)}) > 0$ is a constant depending on the norms of the
noise terms s.t.
\[ \| e^{X^{(2)}_{>}} v \|^2_{L^2 (\mathbb{R}^2)} \leqslant -
   (e^{X^{(2)}_{>}} v, H^{(2)}_{\gg} (e^{X^{(2)}_{>}} v)) . \]
\subsection{The localising operators and the stochastic
terms}\label{sec:locstoc}

Due to the unbounded nature of $\xi$ on $\mathbb{R}^d$, namely it lives only
in a weighted space $\mathcal{C}_{\langle \cdummy \rangle^{- \delta}}^{-
\frac{d}{2} - \varepsilon}$--see Definition \ref{def:besov} for the definition
of (weighted) H{\"o}lder-Besov spaces, we do not proceed directly as on the
torus. Instead we use the decomposition from Gubinelli-Hofmanova
{\cite{GHglobal}} of the noise terms into two parts
\[ \mathcal{C}_{\langle \cdummy \rangle^{- \delta}}^{- \sigma} \ni \Xi =
   \Xi_{\leqslant} + \Xi_{>}  \text{ where } \Xi_{\leqslant} \in L_{\langle
   \cdummy \rangle^{^{- \delta'}}}^{\infty} \text{ and } \Xi_{>} \in
   \mathcal{C}^{- \sigma'}, \text{for } \sigma' > \sigma > 0 \text{ and some }
   \delta' > \delta > 0 \]
i.e. we obtain a part $\Xi_{\leqslant}$ which is regular but unbounded and a
part $\Xi_{>} $which is irregular but ``bounded''. We recall the localisation
operators $\mathcal{U}_{>}$ and $\mathcal{U}_{\leqslant}$ from
{\cite{GHglobal}} defined as
\begin{equation}
  \mathcal{U}_{>} f \assign \underset{k}{\sum} w_k \Delta_{> L_k} f \qquad
  \mathcal{U}_{\leqslant} f \assign \underset{k}{\sum} w_k \Delta_{\leqslant
  L_k} f,
\end{equation}
where $(w_k)$ is a smooth dyadic partition of unity and $\Delta_{> L_k}$ and
$\Delta_{\leqslant L_k}$ are projections on frequencies higher or lower than
$L_k$ respectively. The following is a result from {\cite{GHglobal}} which we
will apply liberally.

\begin{proposition}[Localisation operators, {\cite{GHglobal}}]
  \label{prop:loc}Let $L > 0,$ then there exists a choice of parameters
  $(L_k)$ such that for all $\alpha, \delta, \gamma > 0$ and $a, b \in
  \mathbb{R}$
  \[ \| \mathcal{U}_{>} f \|_{\mathcal{C}_{\langle \cdummy \rangle^a}^{-
     \alpha - \delta}} \lesssim 2^{- \delta L} \| f \|_{\mathcal{C}_{\langle
     \cdummy \rangle^{a - \delta}}^{- \alpha}} \quad \text{and} \qquad \|
     \mathcal{U}_{\leqslant} f \|_{\mathcal{C}_{\langle \cdummy \rangle^{-
     b}}^{- \alpha + \gamma}} \lesssim 2^{\gamma L} \| f
     \|_{\mathcal{C}_{\langle \cdummy \rangle^{- b + \gamma}}^{- \alpha}} \]
\end{proposition}

We give two remarks that one should keep in mind.

\begin{remark}
  Note that the above result is stated only for $f$ with strictly negative
  regularity, however we sometimes apply it to stochastic terms with positive
  regularity (e.g. to $W$ in \eqref{def:W}), meaning of course that the
  decomposition is actually
  \begin{eqnarray*}
    W & = & W_{>} + W_{\leqslant}\\
    & \text{with} & \\
    W_{>} = (1 - \Delta)^{- 1} \mathcal{U}_{>} ((1 - \Delta) W) & \text{and} &
    W_{\leqslant} = (1 - \Delta)^{- 1} \mathcal{U}_{\leqslant} ((1 - \Delta)
    W)
  \end{eqnarray*}
  which has the desired properties.
\end{remark}

\begin{remark}
  The decomposition from Proposition \ref{prop:loc} clearly depends on the
  precise choice of the sequence $(L_k)$ and thus one might wonder whether
  objects like the truncated operators in Definition \ref{def:3ops} are
  actually well-defined. While changing the sequence $(L_k)$ of course changes
  the objects appearing in the truncated operators, importantly the form
  domain $e^{W_{>}} \mathcal{H}^1$ will not change since changing the sequence
  $(L_k) $will result in a different object $\tilde{W}_{>}$ for which one has that
  \[ \tilde{W}_{>} - W_{>}  \text{ is smooth}, \]
  meaning that $e^{W_{>}} \mathcal{H}^1 = e^{\tilde{W}_{>}} e^{W_{>} -
  \tilde{W}_{>}} \mathcal{H}^1 = e^{\tilde{W}_{>}} \mathcal{H}^1 .$ Similarly
  one can see that the definition of the ``full operator'' \eqref{eqn:AHwhole}
  will not change if one changes the sequence $(L_k) .$ Moreover, the
  definition of the truncated operators will only change by a smooth term if
  one changes the localising sequence.
\end{remark}

Now we give a result which says that all the stochastic objects appearing in
the definition of the operators in Definition \ref{def:3ops} actually exist
and are the limits of their periodic counterparts which appear in {\cite{GUZ}}
and thus also the limits of smooth objects.

\begin{theorem}[``Lift'' of the noise in 2- and 3-dimensions]
  \label{thm:noise}For $d = 2, 3$let $\xi^{(d)}$ and $\xi_M^{(d)}$ be the
  spatial white noise on $\mathbb{R}^d$ and the $d -$dimensional torus
  $\mathbb{T}_M^d = (M\mathbb{T})^d$ of size $M > 0$ respectively. Let
  \begin{eqnarray*}
    X^{(d)} \assign (1 - \Delta)^{- 1} \xi^{(d)} & \text{ \ } & X_M^{(d)}
    \assign (1 - \Delta)^{- 1} \xi^{(d)}_M\\
    \xi^{(d)}_{\varepsilon} = \xi^{(d)} \ast \eta^{(d)}_{\varepsilon} &  &
    \xi_{M, \varepsilon}^{(d)} = \xi_M^{(d)} \ast \eta^{(d)}_{\varepsilon}\\
    X_{\varepsilon}^{(d)} \assign (1 - \Delta)^{- 1} \xi^{(d)}_{\varepsilon} &
    & X_{M, \varepsilon}^{(d)} \assign (1 - \Delta)^{- 1} \xi_{M,
    \varepsilon}^{(d)}
  \end{eqnarray*}
  for some smoothing kernels $\eta^{(d)}_{\varepsilon}$.
  \begin{enumerateroman}
    \item In the case $d = 2$ there exist random distributions $: | \nabla
    X_M^{(2)} | :$and  $: | \nabla X^{(2)} | :$s.t.
    \begin{eqnarray*}
      | \nabla X_{M, \varepsilon}^{(2)} |^2 - a^{(2)}_{M, \varepsilon}
      \rightarrow : | \nabla X_M^{(2)} |^2 : & \text{a.s. in}  &
      \mathcal{C}^{- \delta} (\mathbb{T}_M^2)  \text{ as } \varepsilon
      \rightarrow 0\\
      | \nabla X_{\varepsilon}^{(2)} |^2 - a^{(2)}_{\varepsilon} \rightarrow :
      | \nabla X^{(2)} |^2 : & \text{a.s. in}  & \mathcal{C}_{\langle \cdummy
      \rangle^{- \sigma}}^{- \delta} (\mathbb{R}^2) \text{ as } \varepsilon
      \rightarrow 0\\
      & \text{and} & \\
      \nabla X_M \rightarrow \nabla X \text{and } : | \nabla X_M^{(2)} |^2 :
      \rightarrow : | \nabla X^{(2)} |^2 : & \text{a.s. in} &
      \mathcal{C}_{\langle \cdummy \rangle^{- \sigma}}^{- \delta}
      (\mathbb{R}^2)  \text{ as } M \rightarrow \infty,
    \end{eqnarray*}
    for $\delta, \sigma > 0$ and suitably chosen diverging constants
    $a^{(2)}_{M, \varepsilon}, a^{(2)}_{\varepsilon}$.
    
    \item In the case $d = 3$ there exist random distributions
    \[ X_M^{\zzone}, X_M^{\zztwo}, : \left| \nabla X_M^{\zzone} \right|^2 :,
       \nabla X^{(3)}_M \cdummy \nabla X_M^{\zztwo}, X^{\zzone}, X^{\zztwo}, :
       \left| \nabla X^{\zzone} \right|^2 :, \nabla X^{(3)} \cdummy \nabla
       X^{\zztwo} \]
    s.t.
    \begin{eqnarray*}
      X_{M, \varepsilon}^{\zzone} \backassign (1 - \Delta)^{- 1} (| \nabla
      X_{M, \varepsilon}^{(3)} |^2 - a^{(3)}_{M, \varepsilon}) \rightarrow
      X_M^{\zzone} & \text{a.s. in}  & \mathcal{C}^{1 - \delta}
      (\mathbb{T}_M^3)  \text{ as } \varepsilon \rightarrow 0\\
      X_{\varepsilon}^{\zzone} \backassign (1 - \Delta)^{- 1} (| \nabla
      X_{\varepsilon}^{(3)} |^2 - a^{(3)}_{\varepsilon}) \rightarrow
      X^{\zzone} & \text{a.s. in}  & \mathcal{C}_{\langle \cdummy \rangle^{-
      \sigma}}^{1 - \delta} (\mathbb{R}^3)  \text{ as } \varepsilon
      \rightarrow 0\\
      X_M \rightarrow X \text{a.s. in } \mathcal{C}_{\langle \cdummy
      \rangle^{- \sigma}}^{\frac{1}{2} - \delta} (\mathbb{R}^3)  \text{ and }
      X_M^{\zzone} \rightarrow X^{\zzone} & \text{a.s. in} &
      \mathcal{C}_{\langle \cdummy \rangle^{- \sigma}}^{1 - \delta}
      (\mathbb{R}^3)  \text{ as } M \rightarrow \infty\\
      X_{M, \varepsilon}^{\zztwo} \backassign 2 (1 - \Delta)^{- 1} \left(
      \nabla X^{(3)}_{M, \varepsilon} \cdot \nabla X_{M, \varepsilon}^{\zzone}
      \right) \rightarrow X_M^{\zztwo} & \text{a.s. in}  &
      \mathcal{C}^{\frac{3}{2} - \delta} (\mathbb{T}_M^3)  \text{ as }
      \varepsilon \rightarrow 0\\
      X_{\varepsilon}^{\zztwo} \backassign 2 (1 - \Delta)^{- 1} \left( \nabla
      X^{(3)}_{\varepsilon} \cdot \nabla X_{\varepsilon}^{\zzone} \right)
      \rightarrow X^{\zztwo} & \text{a.s. in}  & \mathcal{C}_{\langle \cdummy
      \rangle^{- \sigma}}^{\frac{3}{2} - \delta} (\mathbb{R}^3)  \text{ as }
      \varepsilon \rightarrow 0\\
      \left| \nabla X_{M, \varepsilon}^{\zzone} \right|^2 - b_{M, \varepsilon}
      \rightarrow : \left| \nabla X_M^{\zzone} \right|^2 : & \text{a.s. in}  &
      \mathcal{C}^{- \delta} (\mathbb{T}_M^3)  \text{ as } \varepsilon
      \rightarrow 0\\
      \left| \nabla X_{\varepsilon}^{\zzone} \right|^2 - b_{\varepsilon}
      \rightarrow : \left| \nabla X^{\zzone} \right|^2 : & \text{a.s. in}  &
      \mathcal{C}_{\langle \cdummy \rangle^{- \sigma}}^{- \delta}
      (\mathbb{R}^3)  \text{ as } \varepsilon \rightarrow 0\\
      \nabla X^{(3)}_{M, \varepsilon} \cdummy \nabla X_{M,
      \varepsilon}^{\zztwo} \rightarrow \nabla X^{(3)}_M \cdummy \nabla
      X_M^{\zztwo} & \text{a.s. in}  & \mathcal{C}^{- \frac{1}{2} - \delta}
      (\mathbb{T}_M^3)  \text{ as } \varepsilon \rightarrow 0\\
      \nabla X_{\varepsilon} \cdummy \nabla X_{\varepsilon}^{\zztwo}
      \rightarrow \nabla X^{(3)} \cdummy \nabla X^{\zztwo} & \text{a.s. in}  &
      \mathcal{C}_{\langle \cdummy \rangle^{- \sigma}}^{- \frac{1}{2} -
      \delta} (\mathbb{R}^3)  \text{ as } \varepsilon \rightarrow 0
    \end{eqnarray*}
    and
    \begin{eqnarray*}
      : \left| \nabla X_M^{\zzone} \right|^2 : \rightarrow : \left| \nabla
      X^{\zzone} \right|^2 : & \text{a.s. in } & \mathcal{C}_{\langle \cdummy
      \rangle^{- \sigma}}^{- \delta} (\mathbb{R}^3),\\
      \nabla X^{(3)}_M \cdummy \nabla X_M^{\zztwo} \rightarrow \nabla X^{(3)}
      \cdummy \nabla X^{\zztwo} & \text{a.s. in} & \mathcal{C}_{\langle
      \cdummy \rangle^{- \sigma}}^{1 - \delta} (\mathbb{R}^3)  \text{ as } M
      \rightarrow \infty,
    \end{eqnarray*}
    for $\delta, \sigma > 0.$
  \end{enumerateroman}
  \begin{proof}
    The objects appearing on the whole space are basically the same that
    appear in {\cite{HairerLabbe15}} and {\cite{hairerlabbe3d}} in 2- and 3-d
    respectively. The periodic objects are the same (up to rescaling) that
    appeared in {\cite{GUZ}}. Furthermore, to show the convergence of the
    periodic objects one can proceed as in Section 3.1 in {\cite{GHglobal}}.
  \end{proof}
\end{theorem}

\begin{notation}
  From now on we drop the dimensional index and instead just work in 3
  dimensions, seeing that the two-dimensional case is simpler, meaning we
  write
  \[ X^{(3)} \equiv X \text{etc.} \]
\end{notation}

We need another result which tells us that the product $e^{W_{>}} Z_{>}$ is
(locally) in $\mathcal{C}^{- \frac{1}{2} - \varepsilon}$ for small
$\varepsilon > 0.$

\begin{lemma}
  \label{lem:stocprod}Let $Z, W$ be defined as before and let $\psi$ be
  Lipschitz with compact support, then for any $\varepsilon > 0$ we have
  \[ \| \psi e^{2 W_{>}} Z_{>} \|_{\mathcal{C}^{- \frac{1}{2} - \varepsilon}}
     \leqslant C (\Xi, \psi), \]
  where $C (\Xi, \psi) > 0$ is a constant depending on the norms of the noise
  terms appearing in Theorem \ref{thm:noise} and the bump function $\psi .$
  
  \begin{proof}
    This is essentially Lemma 2.40 from {\cite{GUZ}}. Note that the bump
    function has enough regularity to multiply it with the stochastic terms.
  \end{proof}
\end{lemma}

Lastly we make a simple observation that for functions in weighted $L^p
-$spaces, say $f \in L_{\langle \cdummy \rangle^{- \rho}}^p$ for $\rho > 0$,
one has the following bound for the product with an indicator function of a
ball
\begin{equation}
  \| \chi_{B (R)} f \|_{L^p (\mathbb{R}^d)} \leqslant \| \chi_{B (R)}
  \|_{L_{\langle \cdummy \rangle^{\rho}}^{\infty} (\mathbb{R}^d)} \| f
  \|_{L_{\langle \cdummy \rangle^{- \rho}}^p (\mathbb{R}^d)} \leqslant
  R^{\rho} \| f \|_{L_{\langle \cdummy \rangle^{- \rho}}^p (\mathbb{R}^d)} .
  \label{eqn:Rpol}
\end{equation}
This allows us to localise the ``bulk'' terms at the cost of gaining a large
constant and will be quite useful later in Gronwall-type arguments.

\section{Global space-time solutions for the equation with truncated
noise}\label{sec:trunc}

In this section we prove global-in-time well-posedness for the truncated wave
equation
\begin{eqnarray}
  \partial^2_t u^{(R)} - H_R u^{(R)} & = & - u^{(R)} | u^{(R)} |^2  \text{on }
  [0, T] \times \mathbb{R}^3  \label{eqn:NLWtrunc}\\
  (u^{(R)}, \partial_t u^{(R)}) & = & (u^{(R)}_0, u^{(R)}_1), \nonumber
\end{eqnarray}
for $T > 0,$where $(u^{(R)}_0, u^{(R)}_1) \in \mathcal{D} \left( \sqrt{-
H_{\gg}} \right) \times L^2$. Recall the operators $H_R$ and $H_{\gg}$ defined
in Definition \ref{def:3ops} and the fact that $\mathcal{D} \left( \sqrt{-
H_{\gg}} \right) = e^{W_{>}} \mathcal{H}^1$ as well as the relevant bounds for
$u^{(R)} \in e^{W_{>}} \mathcal{H}^1$
\begin{equation}
  \left\{\begin{array}{lll}
    & | (u^{(R)}, H_R u^{(R)}) | & \leqslant - 2 (u^{(R)}, H_{\gg} u^{(R)}) +
    C (R, \Xi) \| u^{(R)} \|^2_{L^2}\\
    \| u^{(R)} \|^2_{L^2} \leqslant & - (u^{(R)}, H_{\gg} u^{(R)}) & \leqslant
    - 2 (u^{(R)}, H_R u^{(R)}) + C (R, \Xi) \| u^{(R)} \|^2_{L^2} .
  \end{array}\right. \label{eqn:Hbounds}
\end{equation}
\[ \  \]
from Lemma \ref{lem:3ops}.

We firstly observe that the PDE \eqref{eqn:NLWtrunc} admits a conserved energy
(recall that $H_R$ is self-adjoint) denoted by
\begin{eqnarray*}
  E^{(R)} (u^{(R)}) (t) & \assign & \frac{1}{2} (\partial_t u^{(R)} (t),
  \partial_t u^{(R)} (t))_{L^2} - \frac{1}{2} (u^{(R)} (t), H_R u^{(R)}
  (t))_{L^2} + \frac{1}{4} \int_{\mathbb{R}^3} | u^{(R)} (t, x) |^4 d x\\
  & = & E^{(R)} (u^{(R)}) (0)\\
  & = & E^{(R)} ((u^R_0, u^R_1))\\
  & \assign & \frac{1}{2} (u_1^{(R)}, u_1^{(R)})_{L^2} - \frac{1}{2}
  (u_0^{(R)}, H_R u_0^{(R)})_{L^2} + \frac{1}{4} \int_{\mathbb{R}^3} |
  u^{(R)}_0 (x) |^4 d x,
\end{eqnarray*}
see Section 3.3 in {\cite{GUZ}} for a rigorous justification.

\

Inspired by \eqref{eqn:Hbounds} we define the ``rough part'' of the energy to
be
\begin{align}
  E_{\gg} (u^{(R)}) (t) \assign & \frac{1}{2} (\partial_t u^{(R)} (t),
  \partial_t u^{(R)} (t))_{L^2} - \frac{1}{2} (u^{(R)} (t), H_{\gg} u^{(R)}
  (t))_{L^2} + \frac{1}{4} \int_{\mathbb{R}^3} | u^{(R)} (t, x) |^4 d x \\
   \overset{\eqref{eqn:HRH}}{=} & E^{(R)} (u^{(R)}) (t) + \frac{1}{2} \big(
  u^{(R)} (t), \big( \underset{\backassign
  \Xi^R_{\leqslant}}{\underbrace{C_{\gg} (R, \Xi) + \chi_{B (R)}
  Z_{\leqslant}}} \big) u^{(R)} (t) \big)_{L^2} 
\end{align}
But, importantly, it is positive and controls the \tmtextit{energy norm}
\begin{equation}
  \| \partial_t u^{(R)} \|_{L^2} + \left\| \sqrt{- H_{\gg}} u^{(R)}
  \right\|_{L^2} \label{eqn:energybound}
\end{equation}
uniformly in time. Evidently $E_{\gg} (u^{(R)})$ will not be conserved in
time, however we get
\begin{eqnarray*}
  \frac{d}{d t} E_{\gg} (u^{(R)} (t)) & = & \frac{d}{d t} E^{(R)} (u^{(R)})
  (t) + \frac{d}{d t} \frac{1}{2} (u^{(R)} (t), \Xi^{R}_{\leqslant} u^{(R)}
  (t))_{L^2}\\
  & = & (\partial_t u^{(R)} (t), \Xi^R_{\leqslant} u^{(R)} (t))_{L^2}\\
  | \nosymbol \ldots | & \leqslant & \| \Xi^R_{\leqslant} \|_{L^{\infty}} \|
  u^{(R)} (t) \|_{L^2} \| \partial_t u^{(R)} (t) \|_{L^2}\\
  & \leqslant & \frac{1}{2} \| \Xi^R_{\leqslant} \|_{L^{\infty}}  (\| u^{(R)}
  (t) \|^2_{L^2} + \| \partial_t u^{(R)} (t) \|^2_{L^2})\\
  & \leqslant & C \| \Xi^R_{\leqslant} \|_{L^{\infty}} E^{(R)}_{\gg}
  (u^{(R)}) (t),
\end{eqnarray*}
for some universal constant $C > 0$ having used \eqref{eqn:Hbounds} and
\eqref{eqn:energybound} in the last step.

Thus we get an exponential bound for all times by Gronwall, namely
\begin{equation}
  E_{\gg} (u^{(R)} (t)) \leqslant e^{\tilde{C} (\Xi, R) t} E _{\gg}
  (u^{(R)}_0, u^{(R)}_1) \label{eqn:almostbound}
\end{equation}
for some constant $\tilde{C} (\Xi, R) > 0,$recalling that by \eqref{eqn:Rpol}
the norm $\| \Xi^R_{\leqslant} \|_{L^{\infty}}$ grows polynomially in $R.$
Clearly this bound blows up if we take $R \rightarrow \infty$ but for finite
$R$ we will see that this is enough to get global-in-time solutions to
\eqref{eqn:NLWtrunc}.

\begin{theorem}
  \label{thm:trunc}For any $R > 0$ the equation \eqref{eqn:NLWtrunc} is
  globally well-posed. More precisely, for any $T > 0$ and initial data
  $(u^{(R)}_0, u^{(R)}_1) \in \mathcal{D} \left( \sqrt{- H_{\gg}} \right)
  \times L^2$ there exists a unique solution to
  \begin{equation}
    u^{(R)} (t) = \cos ( t \sqrt{- H_{\gg}} ) u^{(R)}_0 +
    \frac{\sin ( t \sqrt{- H_{\gg}} )}{\sqrt{- H_{\gg}}} u^{(R)}_1
    + \int^t_0 \frac{\sin ( (t - s) \sqrt{- H_{\gg}}
    )}{\sqrt{H_{\gg}}} ( {u^{(R)}}^3 (s) + u^{(R)} (s)
    \Xi^R_{\leqslant} ) d s \label{eqn:truncmild}
  \end{equation}
  in $C_{[0, T]} \mathcal{D} \left( \sqrt{- H_{\gg}} \right) \cap C_{[0, T]}^1
  L^2$ which depends continuously on the data.
\end{theorem}

\begin{proof}
  First of all, the fact that \eqref{eqn:truncmild} is the mild formulation of
  \eqref{eqn:NLWtrunc} follows simply by recalling that $H_R = H_{\gg} +
  \Xi^R_{\leqslant}$ and putting the linear term into the nonlinearity of the
  mild formulation.
  
  Next, we define the operator
  \[ \Psi (w) (t) \assign \cos \left( t \sqrt{- H_{\gg}} \right) u^{(R)}_0 +
     \frac{\sin \left( t \sqrt{- H_{\gg}} \right)}{\sqrt{- H_{\gg}}} u^{(R)}_1
     + \int^t_0 \frac{\sin \left( (t - s) \sqrt{- H_{\gg}}
     \right)}{\sqrt{H_{\gg}}} (w^3 (s) + w (s) \Xi^R_{\leqslant}) d s, \]
  for which we have the straightforward bounds
  \begin{eqnarray*}
    \left\| \sqrt{- H_{\gg}} \Psi (w) (t) \right\|_{L^2} & \leqslant & \left\|
    \sqrt{- H_{\gg}} u^{(R)}_0 \right\|_{L^2} + \| u^{(R)}_1 \|_{L^2} +
    \int^t_0 (\| w (s) \|^3_{L^6} + \| \Xi^R_{\leqslant} \|_{L^{\infty}} \| w
    (s) \|_{L^2}) d s\\
    & \lesssim & \| (u^{(R)}_0, u^{(R)}_1) \|_{\mathcal{D} \left( \sqrt{-
    H_{\gg}} \right) \times L^2} + \int^t_0 \left\| \sqrt{- H_{\gg}} w (s)
    \right\|^3_{L^2} d s + t \| \Xi^R_{\leqslant} \|_{L^{\infty}} \| w
    \|_{L_{[0, t]}^{\infty} L^2}\\
    & \lesssim & \| (u^{(R)}_0, u^{(R)}_1) \|_{\mathcal{D} \left( \sqrt{-
    H_{\gg}} \right) \times L^2} + t \left\| \sqrt{- H_{\gg}} w
    \right\|^3_{L_{[0, t]}^{\infty} L^2} + t \| \Xi^R_{\leqslant}
    \|_{L^{\infty}} \| w \|_{L_{[0, t]}^{\infty} L^2}
  \end{eqnarray*}
  having used the embedding
  \begin{equation}
    \mathcal{D} \left( \sqrt{- H_{\gg}} \right) \hookrightarrow L^6,
  \end{equation}
  which simply follows by noting that for $u = e^{W_{>}} v$, $v \in
  \mathcal{H}^1$, we have
  \begin{equation}
    \| u \|_{L^6} \leqslant \| e^{W_{>}} \|_{L^{\infty}} \| v \|_{L^6}
    \lesssim \| v \|_{\mathcal{H}^1} \lesssim \| u \|_{\mathcal{D} \left(
    \sqrt{- H_{\gg}} \right)} .
  \end{equation}
  Thus we may bound, using our almost-conserved energy $E^{(R)}_{\gg}$,
  \begin{align*}
    \underset{0 \leqslant t \leqslant T^{\ast}}{\sup} \left\| \sqrt{- H_{\gg}}
    \Psi (w) (t) \right\|_{L^2}  \lesssim & \| (u^{(R)}_0, u^{(R)}_1)
    \|_{\mathcal{D} \left( \sqrt{- H_{\gg}} \right) \times L^2} +\\&+ T^{\ast}
    e^{\tilde{C} (\Xi, R) T^{\ast}}  \left( (E_{\gg} (u^{(R)}_0,
    u^{(R)}_1))^{\frac{3}{2}} + \| \Xi^R_{\leqslant} \|_{L^{\infty}} (E_{\gg}
    (u^{(R)}_0, u^{(R)}_1))^{\frac{1}{2}} \right)
  \end{align*}
  and analogously
  \begin{align*}
    \| \partial_t \Psi (w) (t) \|_{L^2}  \lesssim & \| (u^{(R)}_0, u^{(R)}_1)
    \|_{\mathcal{D} \left( \sqrt{- H_{\gg}} \right) \times L^2} + \left\|
    \frac{d}{d t} \int^t_0 \frac{\sin \left( (t - s) \sqrt{- H_{\gg}}
    \right)}{\sqrt{H_{\gg}}} (w^3 (s) + w (s) \Xi^R_{\leqslant}) d s
    \right\|_{L^2}\\
     \lesssim & \| (u^{(R)}_0, u^{(R)}_1) \|_{\mathcal{D} \left( \sqrt{-
    H_{\gg}} \right) \times L^2} + \left\| \int^t_0 \cos \left( \left( (t - s)
    \sqrt{- H_{\gg}} \right) \right) w^3 (s) d s \right\|_{L^2}+ \\&+ \left\|
    \int^t_0 \cos \left( \left( (t - s) \sqrt{- H_{\gg}} \right) \right) (w
    (s) \Xi^R_{\leqslant}) d s \right\|_{L^2}\\
     \lesssim & \| (u^{(R)}_0, u^{(R)}_1) \|_{\mathcal{D} \left( \sqrt{-
    H_{\gg}} \right) \times L^2} + \int^t_0 (\| w (s) \|^3_{L^6} + \|
    \Xi^R_{\leqslant} \|_{L^{\infty}} \| w (s) \|_{L^2}) d s.
  \end{align*}
  Thus we have
  \begin{align*}
    \underset{0 \leqslant t \leqslant T^{\ast}}{\sup} \| \partial_t \Psi (w)
    (t) \|_{L^2}  \lesssim & \| (u^{(R)}_0, u^{(R)}_1) \|_{\mathcal{D} \left(
    \sqrt{- H_{\gg}} \right) \times L^2}+\\ &+ T^{\ast} e^{\tilde{C} (\Xi, R)
    T^{\ast}} \left( (E_{\gg} (u^{(R)}_0, u^{(R)}_1))^{\frac{3}{2}} + \|
    \Xi^R_{\leqslant} \|_{L^{\infty}} (E_{\gg} (u^{(R)}_0,
    u^{(R)}_1))^{\frac{1}{2}} \right) .
  \end{align*}
  Hence we can choose a constant
  \[ M = M \left( \| (u^{(R)}_0, u^{(R)}_1) \|_{\mathcal{D} \left( \sqrt{-
     H_{\gg}} \right) \times L^2} \right) > 0 \]
  and a time horizon
  \[ T^{\ast} = T^{\ast} ((E_{\gg} ((u^{(R)}_0, u^{(R)}_1))), \|
     \Xi^R_{\leqslant} \|_{L^{\infty}}), \]
  for which we have
  \begin{align*}
    \left\| \sqrt{- H_{\gg}} (\Psi (w) - \Psi (v)) (t) \right\|_{L^2}  = &
    \left\| \int^t_0 \sin \left( (t - s) \sqrt{- H_{\gg}} \right) (w^3 (s) -
    v^3 (s) + (w (s) - v (s)) \Xi^R_{\leqslant}) \right\|_{L^2}\\
     \leqslant & C T^{\ast} \| \Xi^R_{\leqslant} \|_{L^{\infty}} \| w - v
    \|_{L_{[0, T^{\ast}]}^{\infty} L^2} + T^{\ast} \| w^3 - v^3 \|_{L_{[0,
    T^{\ast}]}^{\infty} L^2}\\
     \leqslant & C T^{\ast} \| \Xi^R_{\leqslant} \|_{L^{\infty}} \| w - v
    \|_{L_{[0, T^{\ast}]}^{\infty} L^2} + \\ &+ T^{\ast} e^{\tilde{C} (\Xi, R)
    T^{\ast}} \| w - v \|_{L_{[0, T^{\ast}]}^{\infty} \mathcal{D} \left(
    \sqrt{- H_{\gg}} \right)} E_{\gg} ((u^{(R)}_0, u^{(R)}_1))\\
     \leqslant & \frac{1}{3} \| w - v \|_{L_{[0, T^{\ast}]}^{\infty}
    \mathcal{D} \left( \sqrt{- H_{\gg}} \right)},
  \end{align*}
  for $w, v$ in the ball of radius $M$ in the space $C_{[0, T^{\ast}]}
  \mathcal{D} \left( \sqrt{- H_{\gg}} \right) \cap C_{[0, T^{\ast}]}^1 L^2
  $and $0 < t \leqslant T^{\ast} .$ Here we have used
  \begin{eqnarray*}
    \| w^3 (s) - v^3 (s) \|_{L^2} & = & \| (w (s) - v (s)) (w^2 (s) + w (s) v
    (s) + v^2 (s)) \|_{L^2}\\
    & \leqslant & 2 \| w (s) - v (s) \|_{L^6} (\| w (s) \|^2_{L^6} + \| v (s)
    \|^2_{L^6})\\
    & \leqslant & C \| w - v \|_{L_{[0, T^{\ast}]}^{\infty} \mathcal{D}
    \left( \sqrt{- H_{\gg}} \right)} {e^{\tilde{C} (\Xi, R) T^{\ast}}} 
    E_{\gg} ((u^{(R)}_0, u^{(R)}_1)) .
  \end{eqnarray*}
  Analogously we prove
  \[ \| \partial_t (\Psi (w) - \Psi (v)) (t) \|_{L^2} \leqslant \frac{1}{3}
     \| w - v \|_{L_{[0, T^{\ast}]}^{\infty} \mathcal{D} \left( \sqrt{-
     H_{\gg}} \right)} \]
  for $w, v$ and $t$ as above.
  
  This gives us a solution to \eqref{eqn:truncmild} up to time $T^{\ast}$
  which lies in $L_{[0, T^{\ast}]}^{\infty} \mathcal{D} \left( \sqrt{-
  H_{\gg}} \right) \cap W_{[0, T^{\ast}]}^{1, \infty} L^2$. In addition,
  Stone's theorem (see Theorem VIII.7 in {\cite{reedsimon1}}) implies that the
  solution is even continuous in time and its derivative is continuous in
  $L^2$.
  
  Lastly we want to globalise this solution, which essentially means that we
  want to resolve the equation on intervals of length $T^{\ast}$ in order to
  obtain a solution in the entire interval $[0, T] .$ In order to do that we
  need to bound the norm of the solution at time $t$ by the initial data and
  the final time $T.$ In fact we get for a solution $u$ the bound
  \begin{eqnarray*}
    \left\| \sqrt{- H_{\gg}} u (t) \right\|_{L^2} + \| \partial_t u (t)
    \|_{L^2} & \lesssim & K \left( \| (u_0, u_1) \|_{\mathcal{D} \left(
    \sqrt{- H_{\gg}} \right) \times L^2} \right) + T e^{\tilde{C} (\Xi, R) T}
    L (E^{(R)}_{\gg} (u_0, u_1), \| \Xi^R_{\leqslant} \|_{L^{\infty}})
  \end{eqnarray*}
  simply by proceeding as above, here $K$ and $L$ denote some constants
  depending on the data polynomially. This bound allows us to choose a global
  $M$ in our fixed point procedure in which we also have a global $T^{\ast}$.
  
  This implies that we can restart the solution until the final time $T >
  0,$which concludes the proof.
\end{proof}

\section{Finite speed of propagation}\label{sec:speed}

We begin this section by giving first of all the classical proof of the finite
speed of propagation for the wave equation which can be found for example in
{\cite{Evans10}}. As we shall see it is not clear whether it can be adapted to
our situation, however it turns out that we can adapt a modified approach
which goes back to Tartar {\cite{tartar}} which we will briefly review.

Consider first the classical linear wave equation(in three dimensions for
definiteness)
\begin{eqnarray*}
  \partial^2_t v - \Delta v & = & 0\\
  (v, \partial_t v) |_{t = 0} & = & (v_0, v_1) \in \mathcal{H}^1
  (\mathbb{R}^3) \times L^2 (\mathbb{R}^3),
\end{eqnarray*}
which has the conserved energy
\[ E (v) \assign \frac{1}{2} \int_{\mathbb{R}^3} | \partial_t v |^2 + |
   \nabla v |^2 . \]
Furthermore, for a space-time point $(t, x)$ we consider the
\tmtextit{backward light cone}
\[ \mathfrak{C}_{(t, x)} \assign \left\{ (s, y) \in \mathbb{R}^3 : 0
   \leqslant s \leqslant t \text{ and } | y - x | \leqslant t - s \right\} . \]
Now, finite speed of propagation means that the solution $v$ at the space-time
point $(t, x)$ depends \tmtextit{only }on the backward light cone
$\mathfrak{C}_{(t, x)}$. In order to make this more quantitive, we define the
\tmtextit{local energy}
\begin{equation}
  e_{(t, x)} (s) \assign \frac{1}{2} \int_{B  (x, t - s)} | \partial_t v (s,
  y) |^2 + | \nabla v (s, y) |^2 d y. \label{eqn:originale}
\end{equation}
A simple computation yields
\begin{align*}
  \frac{d}{d s} e_{(t, x)} (s)  = & - \frac{1}{2} \int_{\partial B  (x, t -
  s)} | \partial_t v (s, y) |^2 + | \nabla v (s, y) |^2 d y + \int_{B  (x, t -
  s)} \partial^2_t v (s, y) \partial_t v (s, y) + \nabla \partial_t v (s, y)
  \nabla v (s, y) d y\\
   = & - \frac{1}{2} \int_{\partial B  (x, t - s)} | \partial_t v (s, y) |^2
  + | \nabla v (s, y) |^2 d y + \int_{\partial B  (x, t - s)} \partial_t v (s,
  y) \nabla v (s, y) d y +\\&+ \int_{B  (x, t - s)} \partial_t v (s, y)
  \underset{= 0}{\underbrace{(\partial^2_t v (s, y) - \Delta v (s, y))}} d y\\
   \leqslant & - \frac{1}{2} \int_{\partial B  (x, t - s)} | \partial_t v
  (s, y) |^2 + | \nabla v (s, y) |^2 d y + \frac{1}{2} \int_{\partial B  (x, t
  - s)} | \partial_t v (s, y) |^2 + | \nabla v (s, y) |^2 d y\\
   = & 0,
\end{align*}
where we have integrated by parts and used Young's inequality in the second
and third step respectively.

Thus we have for $0<s<t$
\[ e_{(t, x)} (s) \leqslant e_{(t, x)} (0) = \frac{1}{2} \int_{B (x, t)} |
   v_1 |^2 + | \nabla v_0 |^2 . \]
In particular this implies that if the initial data $(v_0, v_1)$ are
constantly equal to zero in $B (x, t)$ then the solution $v$ will also be
equal to zero inside the cone $\mathfrak{C}_{(t, x)}$.

We now reformulate the above result in the following way which is due to
Tartar {\cite{tartar}}: Instead of the definition in \eqref{eqn:originale} we
make the modification
\begin{equation}
  e_{(t, x)} (s) \assign \frac{1}{2} \int_{\mathbb{R}^3} \varphi_{t - s, x}
  (y) (| \partial_t v (s, y) |^2 + | \nabla v (s, y) |^2) d y, \label{etartar}
\end{equation}
where $\varphi_{t - s, x}$ is a positive radially symmetric bump function
approximating $\chi_{B  (x, t - s)}$. We repeat the above computation
\begin{align}
  \frac{d}{d s} e_{(t, x)} (s)  = & \frac{1}{2} \int_{\mathbb{R}^3}
  \frac{d}{d s} \varphi_{t - s, x} (y) (| \partial_t v (s, y) |^2 + | \nabla v
  (s, y) |^2) d y +\nonumber\\&+ \int_{\mathbb{R}^3} \varphi_{t - s, x} (y) (\partial^2_t v
  (s, y) \partial_t v (s, y) + \nabla \partial_t v (s, y) \nabla v (s, y)) d y
  \nonumber\\
   = & - \frac{1}{2} \int_{\mathbb{R}^3} \frac{d}{d t} \varphi_{t - s, x}
  (y) (| \partial_t v (s, y) |^2 + | \nabla v (s, y) |^2) d y +\nonumber\\&+
  \int_{\mathbb{R}^3} \varphi_{t - s, x} (y) \partial_t v (s, y) \underset{=
  0}{\underbrace{(\partial^2_t v (s, y) - \Delta v (s, y))}} d y -
  \int_{\mathbb{R}^3} \nabla \varphi_{t - s, x} (y) \partial_t v (s, y) \nabla
  v (s, y) d y \nonumber\\
   \leqslant & \int_{\mathbb{R}^3} \left( | \nabla \varphi_{t - s, x} (y) |
  - \frac{d}{d t} \varphi_{t - s, x} (y) \right) \left( \frac{1}{2} |
  \partial_t v (s, y) |^2 + \frac{1}{2} | \nabla v (s, y) |^2 \right) d y. 
  \label{negativity}
\end{align}
So if we want to re-obtain the same result as before, we should choose
$\varphi_{t - s, x}$ s.t.
\[ | \nabla \varphi_{t - s, x} (y) | - \frac{d}{d t} \varphi_{t - s, x} (y)
   \leqslant 0. \]
We make the following choice. Set $\psi : \mathbb{R} \rightarrow \mathbb{R}_+$
as
\begin{eqnarray*}
  \psi (r) & = & \left\{\begin{array}{ll}
    1 & r \in (- \infty, 0]\\
    1 - r \qquad & r \in [0, 1]\\
    0 & r \in [1, \infty)
  \end{array}\right.,
\end{eqnarray*}
which is Lipschitz and a.e. differentiable with $\psi' \leqslant 0$; then we
consider
\begin{equation}
  \varphi_{(t, x)} (y, s) \assign \psi (| y - x | -\tmmathbf{c} (t - s))
  \label{eqn:defphi}
\end{equation}
which is equal to $1$ for $| y - x | \leqslant \tmmathbf{c} (t - s)$ and $0$
for $| y - x | -\tmmathbf{c} (t - s) \geqslant 1$ and interpolates linearly
inbetween. Here $\tmmathbf{c}> 0$ is a constant we will choose later; It can
be thought of as the speed of propagation. Observe that
\begin{eqnarray}
  | \nabla \varphi_{(t, x)} (y, s) | & = & \left| \frac{1}{\tmmathbf{c}}
  \frac{d}{d s} \psi (| y - x | -\tmmathbf{c} (t - s)) \frac{y - x}{| y - x |}
  \right| \nonumber\\
  & = & \frac{1}{\tmmathbf{c}} \left| \frac{d}{d s} \varphi_{(t, x)} (y, s)
  \right| \nonumber\\
  & = & - \frac{1}{\tmmathbf{c}} \frac{d}{d s} \varphi_{(t, x)} (y, s)
  \nonumber\\
  & = & \frac{1}{\tmmathbf{c}} \frac{d}{d t} \varphi_{(t, x)} (y, s)
  \nonumber\\
  & = & - \psi' (| y - x | -\tmmathbf{c} (t - s)) \nonumber\\
  & = & \chi_{| y - x | -\tmmathbf{c} (t - s) \in [0, 1]}  \label{phider}
\end{eqnarray}
because of the choice of $\psi$. This shows us that the constant
$\tmmathbf{c}$ allows us to make the bound in \eqref{negativity} more
negative.

Note also that we could in principle choose the constant {\tmstrong{$c$}} to
depend on other parameters such as the size of the noise, however it appears
that is actually sufficient for all our purposes to set
\begin{equation}
  \tmmathbf{c}= 2, \label{c=}
\end{equation}
although it seems the computations would still be true for any $\tmmathbf{c}>
1.$

Another thing to note is that if we add a constant quadratic term to the local
energy i.e.
\[ e_{(t, x)} (s) \assign \frac{1}{2} \int_{\mathbb{R}^3} \varphi_{t - s, x}
   (y) (| \partial_t v (s, y) |^2 + | \nabla v (s, y) |^2 + K | v (s, y) |^2)
   d y \]
for $K \geqslant 1$ this leads to the bound
\begin{align*}
  \frac{d}{d s} e_{(t, x)} (s)  = & - \frac{1}{2} \int_{\mathbb{R}^3}
  \frac{d}{d t} \varphi_{t - s, x} (y) (| \partial_t v (s, y) |^2 + | \nabla v
  (s, y) |^2 + K | v (s, y) |^2) d y +\\&+ K \int_{\mathbb{R}^3} \varphi_{t - s,
  x} (y) \partial_t v (s, y) v (s, y) - \int_{\mathbb{R}^3} \nabla \varphi_{t
  - s, x} (y) \partial_t v (s, y) \nabla v (s, y) d y\\
   \leqslant & - \frac{K}{2} \int_{\mathbb{R}^3} \chi_{| y - x |
  -2 (t - s) \in [0, 1]} | v (s, y) |^2 d y + \int_{\mathbb{R}^3}
  \varphi_{t - s, x} (y) \left( \frac{1}{2} | \partial_t v (s, y) |^2 + K^2 |
  v (s, y) |^2 \right) d y\\
   \leqslant&  - \frac{K}{2} \int_{\mathbb{R}^3} \chi_{| y - x |
  -2 (t - s) \in [0, 1]} | v (s, y) |^2 d y + K e_{(t, x)} (s)
\end{align*}
meaning there is a trade-off in that we gain a negative term on the right-hand
side while paying with a ``Gronwall'' term

Note also that this approach has the upside that one does not need to evaluate
anything on the boundary of a ball, as one does if one takes the approach
\eqref{eqn:originale}, which is useful since we are dealing with distributions
for which it is a priori not at all clear how one would do that.

\

Now we are ready to state the first new result which extends the approach we
just described in order to obtain finite speed of propagation for the linear
wave-type equation
\begin{equation}
  \left\{\begin{array}{lll}
    \partial^2_t u - H u & = & 0 \quad \text{on } \mathbb{R}_+ \times
    \mathbb{R}^3\\
    (u, \partial_t u) |_{t = 0} & = & (u_0, u_1)
  \end{array}\right., \label{eqn:Hlinear}
\end{equation}
where $H \text{``=''} \Delta + \xi$ is the full \tmtextit{Anderson
Hamiltonian}.

Since we do not have a direct way of solving \eqref{eqn:Hlinear} (or indeed
making sense of it for now), we instead consider the family of solutions
$u^{(R)}$ to
\begin{equation}
  \left\{\begin{array}{lll}
    \partial^2_t u^{(R)} - H_R u^{(R)}  & = & 0 \quad \text{on } \mathbb{R}_+
    \times \mathbb{R}^3\\
    (u^{(R)}, \partial_t u^{(R)}) |_{t = 0} & = & (u^{(R)}_0, u^{(R)}_1)
  \end{array}\right. \label{eqn:HRlinear},
\end{equation}
recalling Definition \ref{def:3ops} and prove finite speed of propagation for
them which will give us
\[ u^{(R)} = u^{(L)}  \text{ for } R \geqslant L \]
inside a space-time region $G_L$ which is increasing in $L$ and tends to
$\mathbb{R}_+ \times \mathbb{R}^3$ as $L \rightarrow \infty$ as long as their
initial data agree.

\

We introduce a weak formulation for the formal PDE \eqref{eqn:Hlinear} and
the weak formulation \eqref{eqn:HRlinear} will be a suitably truncated version
thereof. We say $u = e^{W_{>}} v$ is a \tmtextit{weak solution }to
\eqref{eqn:Hlinear} if
\[ (v, \partial_t v) |_{t = 0} = (e^{- W_{>}} u_0, e^{- W_{>}} u_1) \]
and
\begin{align}
  (e^{W_{>}} \partial^2_t v, e^{W_{>}} \phi)_{\mathcal{D} \left( \sqrt{-
  H_{\gg}} \right)^{\ast}, \mathcal{D} \left( \sqrt{- H_{\gg}} \right)} = &
  ( H e^{W_{>}} v, e^{W_{>}}\phi)_{\mathcal{D} \left( \sqrt{-
  		H_{\gg}} \right)^{\ast}, \mathcal{D} \left( \sqrt{- H_{\gg}} \right)}
  \nonumber\\
  = & - (e^{{2 W_{>}} } \nabla v, \nabla \phi)_{L^2} + (e^{{2 W_{>}} } Z_{>}
  v, \phi)_{\mathcal{H}^{- 1}, \mathcal{H}^1}\nonumber \\ &+ (e^{{2 W_{>}} } Z_{\leqslant}
  v, \phi)_{L^2}  \label{eqn:vlinear}\\
  & \text{for all compactly supported } \phi \in \mathcal{H}^1 . \nonumber
\end{align}
Analogously we say that $u^{(R)} = e^{W_{>}} v^{(R)}$ is a \tmtextit{weak
solution }to \eqref{eqn:HRlinear} if
\[ (v^{(R)}, \partial_t v^{(R)}) |_{t = 0} = (e^{- W_{>}} u^{(R)}_0, e^{-
   W_{>}} u^{(R)}_1) \]
and
\begin{align}
  (e^{W_{>}} \partial^2_t v^{(R)}, e^{W_{>}} \phi)_{\mathcal{D} \left( \sqrt{-
  H_{\gg}} \right)^{\ast}, \mathcal{D} \left( \sqrt{- H_{\gg}} \right)} = &
  ( H_R e^{W_{>}} v^{(R)}, e^{W_{>}}\phi)_{\mathcal{D} \left( \sqrt{-
  		H_{\gg}} \right)^{\ast}, \mathcal{D} \left( \sqrt{- H_{\gg}} \right)}
  \nonumber\\
  = & - (e^{{2 W_{>}} } \nabla v^{(R)}, \nabla \phi)_{L^2} + (e^{{2 W_{>}} }
  Z_{>} v^{(R)}, \phi)_{\mathcal{H}^{- 1}, \mathcal{H}^1} + \nonumber \\ &+(e^{{2 W_{>}} }
  \chi_{B (R)} Z_{\leqslant} v^{(R)}, \phi)_{L^2}  \label{eqn:vRlinear}\\
  & \text{for all compactly supported } \phi \in \mathcal{H}^1 . \nonumber
\end{align}
\begin{remark}
  Note that the space $\mathcal{D} \left( \sqrt{- H_{\gg}} \right)^{\ast}$,
  which is the dual of the energy space, is the natural space of the terms
  $\partial^2_t u$ and $\partial^2_t u^{(R)}$ and indeed one can readily show
  that the solutions from Section \ref{sec:trunc} satisfy this property.
\end{remark}

Now we give the main result which says that the linear equation has finite
speed of propagation implying that solutions to \eqref{eqn:vRlinear} are
actually local solutions to \eqref{eqn:vlinear}.

\begin{theorem}
  \label{thm:linear}Let $u^{(R)}, u^{(L)}$ be solutions to
  \begin{eqnarray*}
    \partial^2_t u^{(i)} - H_i u^{(i)} & = & 0\\
    (u^{(i)}, \partial_t u^{(i)}) |_{t = 0} & = & (u^{(i)}_0, u^{(i)}_1)
  \end{eqnarray*}
  for $i = R, L$ and $R \geqslant L \gg 0$. Moreover we choose the initial
  data to satisfy $(u^{(i)}_0, u^{(i)}_1) \in \mathcal{D} \left( \sqrt{-
  H_{\gg}} \right) \times L^2$ and
  \[ (u^{(R)}_0, u^{(R)}_1) = (u^{(L)}_0, u^{(L)}_1)  \text{ on } B (2 L + 1) .
  \]
  Then we have
  \begin{equation}
    u^{(R)} (t, x) = u^{(L)} (t, x)  \text{for all}  (t, x)  \text{s.t. }
    \mathfrak{C}_{(t, x)} \subset \left[ 0, \frac{L}{2} \right] \times B (L),
  \end{equation}
  where the backward light-cone $\mathfrak{C}_{(t, x)}$ at the space-time
  point $(t, x)$ is defined as
  \begin{equation}
    \mathfrak{C}_{(t, x)} \assign \left\{ (s, y) \in \mathbb{R} \times
    \mathbb{R}^3 : 0 \leqslant s \leqslant t \text{and} \quad y \in B (x, 2 (t
    - s)) \right\} .
  \end{equation}
  Moreover, we have the following local bounds for the solutions
  \[ e_{(t, x)}^{(i)} (s) \leqslant C (\Xi, i) e_{(t, x)}^{(i)} (0) = C (\Xi,
     i) \int_{\mathbb{R}^d} \varphi_{(t, x)} (0) \frac{1}{2} e^{2 W_{>}} (|
     v^{(i)}_1 |^2 + | \nabla v^{(i)}_0 |^2 +\mathbf{C} (\Xi, i) | v^{(i)}_0
     |^2 - | v^{(i)}_0 |^2 Z_{>}) \]
  for suitable constants $C (i, \Xi), \mathbf{C} (\Xi, i) > 0$, having defined
  the exponentially transformed solution and initial data as
  \begin{equation}
    v^{(i)} \assign e^{- W_{>}} u^{(i)}  \text{ and } v_j^{(i)} \assign e^{-
    W_{>}} u_j^{(i)}  \text{ for } i = L, R \text{ and } j = 0, 1
    \label{def:linvi}
  \end{equation}
  and the appropriate local energy quantity as
  \begin{equation}
    e_{(t, x)}^{(i)} (s) \assign \int_{\mathbb{R}^d} \varphi_{(t, x)} (s)
    \frac{1}{2} e^{2 W_{>}} (| \partial_t v^{(i)} (s) |^2 + | \nabla v^{(i)}
    (s) |^2 +\mathbf{C} (\Xi, i) | v^{(i)} (s) |^2 - | v^{(i)} (s) |^2 Z_{>}),
    \label{eqn:locen}
  \end{equation}
  for $i = L, R.$ The bump function $\varphi$ is the one defined in
  \eqref{eqn:defphi} setting $\tmmathbf{c}= 2$.
\end{theorem}

\begin{proof}
  We follow the general method of Tartar, see {\cite{tartar}}, which was
  sketched above. Using the exponential transform, we rewrite the equations
  for $u^{(R)}$ and $u^{(L)}$ instead as equations for $v^{(R)}$ and $v^{(L)}$
  which are given by \eqref{eqn:vRlinear} in weak form.
  
  We now consider the local energy quantity \eqref{eqn:locen} which is of
  course inspired by \eqref{etartar} and suppress the $(t, x)$ and the $i$
  dependence of $e$ for ease of notation. Moreover, we have added a (large)
  $L^2$ term --namely $\mathbf{C} (\Xi, i)$ -- which does not come from the
  equation but makes it uniformly positive i.e we have
  \[ e (s) = 0 \quad \text{implies } v^{(i)} = 0 \text{ in }
     \textrm{\tmop{supp}} (\varphi_{(t, x)} (s)), \]
  see Definition \ref{def:3ops}.
  
  Note that this is for now only formal, since the term
  \[ \int_{\mathbb{R}^3} \varphi_{(t, x)} (s) \frac{1}{2} e^{2 W_{>}} |
     v^{(i)} (s) |^2 Z_{>} \]
  is not an honest integral but rather should be thought of as a pairing like
  \[ \left( \varphi_{(t, x)} (s) \frac{1}{2} e^{2 W_{>}} Z_{>}, | v^{(i)} (s)
     |^2 \right)_{{B_{\infty . \infty}^{- \frac{1}{2} - \varepsilon}} , B_{1,
     1}^{\frac{1}{2} + \varepsilon}} \]
  for $\varepsilon > 0$ small, see Lemma \ref{lem:stocprod} for why the first
  term is in fact in \ ${B_{\infty . \infty}^{- \frac{1}{2} - \varepsilon}} 
  =\mathcal{C}^{- \frac{1}{2} - \varepsilon} .$ For the right-hand side we
  invoke Lemmas \ref{lem:leib} and \ref{lem:embedding} in order to bound
  \begin{align}
    {\| | v^{(i)} (s) |^2 \|_{B_{1, 1}^{\frac{1}{2} + \varepsilon}}} 
    \leqslant C \| v^{(i)} (s) \|_{L^2} \| v^{(i)} (s)
    \|_{\mathcal{H}^{\frac{1}{2} + 2 \varepsilon}}  \leqslant & C \| v^{(i)}
    (s) \|^2_{\mathcal{H}^{\frac{3}{4}}} \nonumber\\
     \leqslant & \delta \| e^{W_{>}} \nabla v^{(i)} (s) \|^2_{L^2} + C
    (\delta, \Xi) \| e^{W_{>}} v^{(i)} (s) \|^2_{L^2},  \label{lowerorder1}
  \end{align}
  for small $\delta > 0,$ interpolating in the $\mathcal{H}^{\sigma} -$scale
  and applying Young's inequality, showing that this term is not only
  well-defined but also ``lower-order'' with respect to the gradient term. In
  light of this computation, we will continue to make this mild abuse of
  notation.
  
  Analogously to the strategy above, we compute the derivative $\frac{d}{d s}
  e (s)$ in order to make a Gronwall argument. This yields
  \begin{align}
    \frac{d}{d s} e (s)  = & \int_{\mathbb{R}^3} \frac{d}{d s} \varphi_{(t,
    x)} (s) \frac{1}{2} e^{2 W_{>}} (| \partial_t v^{(i)} (s) |^2 + | \nabla
    v^{(i)} (s) |^2 + | v^{(i)} (s) |^2 (\mathbf{C} (\Xi, i) - Z_{>})) +\nonumber \\&+
    \int_{\mathbb{R}^3} \varphi_{(t, x)} (s) e^{2 W_{>}} \partial_t v^{(i)}
    (s) (\partial^2_t v^{(i)} (s) - v^{(i)} (s) Z_{>} - \chi_{B (R)}
    Z_{\leqslant} v^{(i)} (s) + \chi_{B (R)} Z_{\leqslant} v^{(i)} (s)) +\nonumber \\&+
    \int_{\mathbb{R}^3} \varphi_{(t, x)} (s) e^{2 W_{>}} \nabla v^{(i)} (s)
    \nabla \partial_t v^{(i)} (s) +\mathbf{C} (\Xi, R) \int_{\mathbb{R}^3}
    \varphi_{(t, x)} (s) e^{2 W_{>}} \partial_t v^{(i)} (s) 
    v^{(i)} (s) \nonumber\\
     =  &- \int_{\mathbb{R}^3} 2 \chi_{| \cdummy - x | - 2 (t - s) \in [0,
    1]} \frac{1}{2} e^{2 W_{>}} (| \partial_t v^{(i)} (s) |^2 + | \nabla
    v^{(i)} (s) |^2 + | v^{(i)} (s) |^2 (\mathbf{C} (\Xi, i) - Z_{>}))+\nonumber \\& +
    \int_{\mathbb{R}^3} \varphi_{(t, x)} (s) e^{2 W_{>}} \partial_t v^{(i)}
    (s) \underset{= 0}{\underbrace{(\partial^2_t v^{(i)} (s) - e^{- W_{>}} H_R
    v^{(i)} (s))}} +\nonumber \\&+ \int_{\mathbb{R}^3} \varphi_{(t, x)} (s) e^{2 W_{>}}
    \partial_t v^{(i)} (s) (\chi_{B (i)} Z_{\leqslant} +\mathbf{C} (\Xi, i))
    v^{(i)} (s) - \int_{\mathbb{R}^3} \nabla \varphi_{(t, x)} (s) e^{2 W_{>}}
    \nabla v^{(i)} (s) \partial_t v^{(i)} (s) \nonumber\\
     \leqslant & - \int_{\mathbb{R}^3} \chi_{| \cdummy - x | - 2 (t - s) \in
    [0, 1]} e^{2 W_{>}} (| \partial_t v^{(i)} (s) |^2 + | \nabla v^{(i)} (s)
    |^2 + | v^{(i)} (s) |^2 (\mathbf{C} (\Xi, i) - Z_{>})) +\nonumber \\&+
    \int_{\mathbb{R}^3} | \nabla \varphi_{(t, x)} (s) | e^{2 W_{>}} \left(
    \frac{1}{4} | \nabla v^{(i)} (s) |^2 + | \partial_t v^{(i)} (s) |^2
    \right) +\nonumber \\&+ (\| \chi_{B (i)} Z_{\leqslant} \|_{L^{\infty}} +\mathbf{C} (\Xi,
    i)) \int_{\mathbb{R}^3} \varphi_{(t, x)} (s) e^{2 W_{>}} \frac{1}{2} (|
    \partial_t v^{(i)} (s) |^2 + | v^{(i)} (s) |^2) \nonumber\\
     = & - \int_{\mathbb{R}^3} \chi_{| \cdummy - x | - 2 (t - s) \in [0, 1]}
    e^{2 W_{>}} \left( \frac{3}{4} | \nabla v^{(i)} (s) |^2 + (\mathbf{C}
    (\Xi, i) - Z_{>}) | v^{(i)} (s) |^2 \right) +\nonumber \\&+ (\| \chi_{B (i)}
    Z_{\leqslant} \|_{L^{\infty}} +\mathbf{C} (\Xi, i)) e (s), 
    \label{eqn:3/4}
  \end{align}
  where we have used \eqref{phider}, integration by parts and Young's
  inequality.
  
  Now we want to conclude by arguing that the term
  \begin{equation}
    \int_{\mathbb{R}^3} \chi_{| \cdummy - x | - 2 (t - s) \in [0, 1]} e^{2
    W_{>}} Z_{>} | v^{(i)} |^2 (s) \label{eqn:Zterm}
  \end{equation}
  can be absorbed by the other two using the fact that we may freely choose
  $\mathbf{C} (\Xi, i)$ depending on the norm of $\Xi$ by an argument similar
  to \eqref{lowerorder1}.
  
  We however need one trick to proceed, since the term in the right-hand side
  of the bound \eqref{lowerorder1} can not be controlled by the terms we have.
  Recall that we have a bounded restriction and an extension operator on Besov
  spaces $\cdummy |_A$ and $E_A$ for nice sets $A \subset \mathbb{R}^d$, see
  Proposition \ref{prop:ext}.
  
  Generally we may bound the product with an indicator function in the
  following way using Lemma \ref{lem:leib} and Lemma \ref{lem:embedding} and
  the restriction/extensions from Proposition \ref{prop:ext}
  \begin{eqnarray*}
    \| \chi_A f \|_{B_{1, 1}^{\alpha} (\mathbb{R}^3)} & = & \| \chi_A E_A
    (f|_A) \|_{B_{1, 1}^{\alpha} (\mathbb{R}^3)}\\
    & \lesssim & \| \chi_A \|_{B_{p, 1}^{\alpha} (\mathbb{R}^3)} \| E_A f|_A
    \|_{L^q (\mathbb{R}^3)} + \| E_A f|_A \|_{B_{r, 1}^{\alpha}
    (\mathbb{R}^d)} \| \chi_A \|_{L^{\rho} (\mathbb{R}^3)}\\
    & \lesssim & \| E_A f|_A \|_{B_{r, 1}^{\alpha} (\mathbb{R}^d)} (\| \chi_A
    \|_{B_{p, 1}^{\alpha} (\mathbb{R}^d)} + \| \chi_A \|_{L^{\rho}
    (\mathbb{R}^3)})\\
    & \lesssim & \| f|_A \|_{B_{r, 1}^{\alpha} (A)} (\| \chi_A \|_{B_{p,
    1}^{\alpha} (\mathbb{R}^d)} + \| \chi_A \|_{L^{\rho} (\mathbb{R}^3)})\\
    & \text{for} & \\
    1 & = & \frac{1}{p} + \frac{1}{q} = \frac{1}{r} + \frac{1}{\rho},\\
    q & < & \frac{3 r}{3 - \alpha r},
  \end{eqnarray*}
  where the terms involving $\chi_A$ are finite for $A \subset \mathbb{R}^3$
  with finite perimeter and $\alpha < \frac{1}{p}$ by Lemma \ref{lem:charreg}
  and Lemma \ref{lem:embedding}.
  
  \
  
  In fact, we bound the term \eqref{eqn:Zterm} as follows. We set
  \[ A \assign \{ y \in \mathbb{R}^3 : | y - x | - 2 (t - s) \in [0, 1] \},
  \]
  and taking suitably small $\varepsilon, \tilde{\varepsilon}, \delta > 0$ we
  proceed as above
  \begin{eqnarray}
    | (\chi_A | v^{(i)} (s) |^2, e^{2 W_{>}} Z_{>}) | & \leqslant & \| e^{2
    W_{>}} Z_{>} \|_{B_{\infty, \infty}^{- \frac{1}{2} -
    \frac{\varepsilon}{4}}} \| \chi_A | v^{(i)} (s) |^2 \|_{B_{1,
    1}^{\frac{1}{2} + \frac{\varepsilon}{4}}} \nonumber\\
    & = & \| e^{2 W_{>}} Z_{>} \|_{B_{\infty, \infty}^{- \frac{1}{2} -
    \frac{\varepsilon}{4}}} \| \chi_A E_A (| v^{(i)} (s) |^2) |_A \|_{B_{1,
    1}^{\frac{1}{2} + \frac{\varepsilon}{4}}} \nonumber\\
    & \leqslant & C \| e^{2 W_{>}} Z_{>} \|_{B_{\infty, \infty}^{-
    \frac{1}{2} - \frac{\varepsilon}{4}}} \Big( \| \chi_A \|_{B_{\frac{2}{1 +
    \varepsilon}, 1}^{\frac{1}{2} + \frac{\varepsilon}{4}} (\mathbb{R}^3)} \|
    E_A (| v^{(i)} (s) |^2) |_A \|_{L^{\frac{2}{1 - \varepsilon}}} +\\
    &&+ \| E_A (|
    v^{(i)} (s) |^2)_A \|_{B_{\frac{1}{1 - \delta}, 1}^{\frac{1}{2} +
    \frac{\varepsilon}{4}} (\mathbb{R}^3)} \| \chi_A \|_{L^{\frac{1}{\delta}}}
    \Big) \nonumber\\
    & \leqslant & C (A, \Xi) \left( \| E_A (| v^{(i)} (s) |^2) |_A
    \|_{L^{\frac{2}{1 - \varepsilon}} (\mathbb{R}^3)} + \| E_A (| v^{(i)} (s)
    |^2) |_A \|_{B_{\frac{1}{1 - \delta}, 1}^{\frac{1}{2} +
    \frac{\varepsilon}{4}} (\mathbb{R}^3)} \right) \nonumber\\
    & \leqslant & C (A, \Xi) \left( \| v^{(i)} (s) |_A  \|^2_{L^{\frac{4}{1 -
    \varepsilon}} (A)} + \| (v^{(i)} (s))^2 \|_{B_{\frac{1}{1 - \delta},
    1}^{\frac{1}{2} + \frac{\varepsilon}{4}} (A)} \right) \nonumber\\
    & \leqslant & C (A, \Xi) \left( \| v^{(R)} (s) \|^2_{L^{\frac{4}{1 -
    \varepsilon}} (A)} + \| v^{(R)} (s) \|^2_{B_{\frac{2}{1 - \delta},
    2}^{\frac{1 + \varepsilon}{2}} (A)} \right) \nonumber\\
    & \leqslant & C (A, \Xi) \left( \| v^{(i)} (s) 
    \|^2_{{\mathcal{H}^{\frac{3}{4} + \tilde{\varepsilon}}}  (A)} + \| v^{(i)}
    (s) \|^2_{\mathcal{H}^{\frac{1}{2} + \tilde{\varepsilon}} (A)} \right)
    \nonumber\\
    & \leqslant & \frac{1}{4 \| e^{- 2 W_{>}} \|_{L^{\infty}}} \int_A |
    \nabla v^{(i)} (s) |^2 + C (A, \Xi) \int_A | v^{(i)} (s) |^2 \nonumber\\
    & \leqslant & \frac{1}{4} \int_A e^{2 W_{>}} | \nabla v^{(i)} (s) |^2 + C
    (A, \Xi) \int_A e^{2 W_{>}} | v^{(i)} (s) |^2 .  \label{eqn:1/8}
  \end{eqnarray}
  Now we observe that the dependence of the constant in $A$ can be chosen
  uniformly in $s \in [0, t]$ and does not depend on the point $x$ at all.
  Instead this will result in an $i$ dependent constant. More precisely, by
  translation invariance one can see that the norms do not depend on the
  spatial variable $x$ and to see that one may choose it independent of the
  time integration parameter $s$ one observes that the function $s \rightarrow
  \| \chi_{| \cdummy - x | - 2 (t - s) \in [0, 1]} \|_{B_{\frac{2}{1 +
  \varepsilon}, 1}^{\frac{1}{2} + \frac{\varepsilon}{4}} (\mathbb{R}^3)}$ is
  continuous on $s \in [0, t]$ and bounded at the end points(see Lemma
  \ref{lem:charreg}) hence is bounded on the whole interval. Lastly, since we
  only consider times $t \leqslant i$ we have that the constant can be chosen
  to depend on $R$ and we relabel it as
  \[ c (\Xi, i) = C (A, \Xi) . \]

  If we insert the bound \eqref{eqn:1/8} into \eqref{eqn:3/4} we get after
  choosing the constant $\mathbf{C} (\Xi, i)$ sufficiently large
  \begin{eqnarray}
    \frac{d}{d s} e (s) & \leqslant & \int_{\mathbb{R}^d} \chi_{| \cdummy - x
    | - 2 (t - s) \in [0, 1]} \frac{1}{2} e^{2 W_{>}} \Bigg(
    \underset{\leqslant 0}{\underbrace{\left( \frac{1}{4} - \frac{3}{4}
    \right)}} | \nabla v^{(i)} (s) |^2 + \underset{\overset{!}{\leqslant}
    0}{\underbrace{(c (\Xi, i) -\mathbf{C} (\Xi, i))}} | v^{(R)} (s) |^2
    \Bigg)+ \nonumber\\ &&+ (\| \chi_{B (R)} Z_{\leqslant} \|_{L^{\infty}} +\mathbf{C} (\Xi,
    i)) e (s) \nonumber\\
    & \leqslant & (\| \chi_{B (R)} Z_{\leqslant} \|_{L^{\infty}} +\mathbf{C}
    (\Xi, i)) e (s)  \label{eqn:ebound}
  \end{eqnarray}
  which by Gronwall implies
  \[ e (s) \leqslant C (i, \Xi) e (0) = C (i, \Xi) \int_{\mathbb{R}^d}
     \varphi_{(t, x)} (0) \frac{1}{2} e^{2 W_{>}} (| v^{(i)}_1 |^2 + | \nabla
     v^{(i)}_0 |^2 +\mathbf{C} (\Xi, i) | v^{(i)}_0 |^2 - | v^{(i)}_0 |^2
     Z_{>}) . \]
  This in particular implies that $v^{(i)}$ is controlled in points inside the
  backwards light cone by the initial conditions in the support of
  $\varphi_{(t, x)} (0)$, i.e. a ball around $x.$ Moreover this implies for
  two different parameters
  \[ 0 \ll L \leqslant R \]
  that the solutions $v^{(L)}$ and $v^{(R)}$ to \eqref{eqn:HRlinear} with the
  same initial data $(v_0, v_1)$ actually agree in the backward light-cones
  which are contained in $[0, \frac{L}{2}] \times B ( L)$.
  
  In order to make this precise, we observe that the difference
  \[ d \assign v^{(L)} - v^{(R)} \]
  solves the equation
  \begin{eqnarray*}
    (\partial^2_t - H_L) d & = & \underset{= 0 \text{in } B_L
    (0)}{\underbrace{(\chi_{B (R)} Z_{\leqslant} - \chi_{B (L)}
    Z_{\leqslant})}} v^{(R)}\\
    (d, \partial_t d) |_{t = 0} & = & (0, 0) .
  \end{eqnarray*}
  Thus the above argument applied to points $(t, x)$ for which $\tmop{supp}
  (\varphi_{(t, x)} (s)) \subset B_L$ for all $0 \leqslant s \leqslant t$
  gives
  \begin{eqnarray*}
    0 & \leqslant & \int_{\mathbb{R}^d} \varphi_{(t, x)} (s) e^{2 W_{>}} | d
    (s) |^2\\
    & \leqslant & \int_{\mathbb{R}^d} \frac{1}{2} e^{2 W_{>}} \varphi_{(t,
    x)} (s) (| \partial_t d (s) |^2 + | \nabla d (s) |^2 + | d (s) |^2
    (\mathbf{C} (\Xi, R) - Z_{>}))\\
    & \leqslant & C (\Xi, R) \int_{\mathbb{R}^d} \varphi_{(t, x)} (0) (|
    \partial_t d (0) |^2 + | \nabla d (0) |^2 + (\mathbf{C} (\Xi, R) - Z_{>})
    | d (0) |^2)\\
    & = & 0,
  \end{eqnarray*}
  which implies that $d \equiv 0$ in that region. This finishes the proof.
\end{proof}

\begin{remark}\label{rem:max}
Shortly before completion, it was pointed out to the author by Massimilano Gubinelli that one could alternatively prove the finite speed of propagation of the multiplicative stochastic wave equation by approximating the equation by regularising the noise and localising\begin{align}
	\partial_t^2u-H^{loc}_\varepsilon u=0 \label{eqn:wavemax} \\
	(u,\partial_tu)|_{t=0}=(u_0^\varepsilon,u_1^\varepsilon),
\end{align}
where $H^{loc}_\varepsilon=\Delta +\underbrace{\Xi^{loc}_\varepsilon}_{\text{smooth} }$ is some suitable regular and localised approximation to the Anderson Hamiltonian which in particular should be self-adjoint and semibounded. This equation then has unit speed of propagation by the classical proof above, meaning for every $t,\varepsilon>0$ and $x\in\mathbb{R}^3$ 
\begin{align}
\chi_{B  (x, t - s)} S^{loc}_\varepsilon (s) (u_0^\varepsilon,u_1^\varepsilon)= S^{loc}_\varepsilon (s) (\chi_{B  (x, t)}u_0^\varepsilon,\chi_{B  (x, t )}u_1^\varepsilon) \text{ for any } 0<s<t, \label{eqn:finitemax}
\end{align}
where $ S^{loc}_\varepsilon$ is the propagator of the equation \eqref{eqn:wavemax}. If one then proves the strong resolvent convergence of the operators $H^{loc}_\varepsilon$ to a localised version of the Anderson Hamiltonian as $\varepsilon\to0$, one gets that the associated propagators converge strongly in $L^2$, cf Section 3.3 in \cite{GUZ}. This convergence would imply that the identity \eqref{eqn:finitemax} passes to the limit as $\varepsilon\to0$ meaning we have unit speed of propagation with a localisation which can be removed since we are interested in a local property.\\
This approach, however, does not immediately give us bounds on the local energies like in the approach in the current work. Also it is not immediately clear how one would prove the analogous result in the nonlinear case.
\end{remark}

\section{Putting it all together}\label{sec:cubicfinite}

Finally we want to apply Theorem \ref{thm:linear} also to semilinear wave
equations which are formally
\begin{equation}
  \left\{\begin{array}{lll}
    \partial^2_t u - H u &= & - u^3  \text{ on } \mathbb{R}_+ \times
    \mathbb{R}^3\\
    (u, \partial_t u)_{t = 0} &= & (u_0, u_1)  
  \end{array}\right. \label{eqn:cubicformal}
\end{equation}
whose solutions should be suitable limits of the solutions $u^{(R)}$ to
\eqref{eqn:NLWtrunc}.

Analogously to before, we introduce the weak formulation of both the ``full''
PDE and its truncated version, which we have solved in Theorem
\ref{thm:trunc}.

We say $u = e^{W_{>}} v$ is a \tmtextit{weak solution }to
\eqref{eqn:cubicformal} if
\begin{align}
  (e^{W_{>}} \partial^2_t v, e^{W_{>}} \phi)_{\mathcal{D} \left( \sqrt{-
  H_{\gg}} \right)^{\ast}, \mathcal{D} \left( \sqrt{- H_{\gg}} \right)} = &
  ( H e^{W_{>}} v,e^{W_{>}} \phi)_{\mathcal{D} \left( \sqrt{-
  		H_{\gg}} \right)^{\ast}, \mathcal{D} \left( \sqrt{- H_{\gg}} \right)} - (e^{4
  W_{>}} v^3, \phi)_{L^2} \nonumber\\
  = & - (e^{{2 W_{>}} } \nabla v, \nabla \phi)_{L^2} + (e^{{2 W_{>}} } Z_{>}
  v, \phi)_{\mathcal{H}^{- 1}, \mathcal{H}^1} +\nonumber \\&+ (e^{{2 W_{>}} } Z_{\leqslant}
  v, \phi)_{L^2} - (e^{4 W_{>}} v^3, \phi)_{L^2}  \label{eqn:weakcubic}\\
  & \text{for all compactly supported } \phi \in \mathcal{H}^1 . \nonumber\\
  & \text{and} \nonumber\\
  (v, \partial_t v)_{t = 0} = & (e^{- W_{>}} u_0, e^{- W_{>}} u_1) \nonumber
\end{align}
Analogously we say that $u^{(R)} = e^{W_{>}} v^{(R)}$ is a \tmtextit{weak
solution }to \eqref{eqn:NLWtrunc}  if
\begin{align}
  (e^{W_{>}} \partial^2_t v^{(R)}, e^{W_{>}} \phi)_{\mathcal{D} \left( \sqrt{-
  H_{\gg}} \right)^{\ast}, \mathcal{D} \left( \sqrt{- H_{\gg}} \right)} = &
  ( H_R e^{W_{>}} v^{(R)}, e^{W_{>}}\phi)_{\mathcal{D} \left( \sqrt{-
  		H_{\gg}} \right)^{\ast}, \mathcal{D} \left( \sqrt{- H_{\gg}} \right)}- (e^{4
  	W_{>}} v^3, \phi)_{L^2}
  \nonumber\\
  = & - (e^{{2 W_{>}} } \nabla v^{(R)}, \nabla \phi)_{L^2} + (Z_{>} v^{(R)},
  \phi)_{\mathcal{H}^{- 1}, \mathcal{H}^1}+\nonumber \\ & + (\chi_{B (0, R)} Z_{\leqslant}
  v^{(R)}, \phi)_{L^2}- (e^{4 W_{>}} v^3, \phi)_{L^2} 
  \\
  & \text{for all compactly supported } \phi \in \mathcal{H}^1 . \nonumber
  \\
  & \text{and} \nonumber\\
  (v^{(R)}, \partial_t v^{(R)})_{t = 0} = & (e^{- W_{>}} u^{(R)}_0, e^{-
  W_{>}} u^{(R)}_1) \nonumber
\end{align}
We wish to prove an analogous bound to \eqref{eqn:ebound} for the nonlinear
equation. In fact we get the following result which extends the finite speed
of propagation argument to the semilinear case. Since the nonlinearity is
controlled by the energy, this is essentially like Theorem \ref{thm:linear}
with some modifications.

\begin{theorem}[Finite speed of propagation for the cubic multiplicative
stochastic wave equation]
  \label{thm:finitecubic}
  
  Let $R \geqslant L \gg 0$ and $u^{(i)} = e^{W_{>}} v^{(i)}$ be the solutions
  to
  \begin{eqnarray}
    \partial^2_t u^{(i)} - H_i u^{(i)} & = & - u^{(i)} | u^{(i)} |^2  \text{on
    } [0, T] \times \mathbb{R}^3  \label{eqn:NLWthm}\\
    (u, \partial_t u) & = & (u_0, u_1) \in \mathcal{D} \left( \sqrt{- H_{\gg}}
    \right) \times L^2, \nonumber
  \end{eqnarray}
  from Theorem \ref{thm:trunc} for some T>0 and $i = R, L$. We set
  \[ e^{(i)} (s) \assign \int_{\mathbb{R}^d} \varphi_{(t, x)} (s) \frac{1}{2}
     e^{2 W_{>}} (| \partial_t v^{(i)} (s) |^2 + | \nabla v^{(i)} (s) |^2 + |
     v^{(i)} (s) |^2 (\tmmathbf{C} (\Xi, i) - Z_{>})) \]
  and
  \[ e (s) \assign \int_{\mathbb{R}^d} \varphi_{(t, x)} (s) \frac{1}{2} e^{2
     W_{>}} (| \partial_t b (s) |^2 + | \nabla b (s) |^2 + | b (s) |^2
     (\tmmathbf{C} (\Xi, i) - Z_{>})), \]
  where $b \assign v^{(R)} - v^{(L)}$for large constants $\tmmathbf{C} (\Xi,
  i) > 0$ chosen below.
  
  Then \ there exist constants $c (i, \Xi, u^i_0, u^i_1), c (\Xi, L, R, u^R_0,
  u^R_1, u^L_0, u^L_1) > 0$ for which the bounds
  \begin{align}
    e^{(i)} (s)  \leqslant & c (i, \Xi, u^i_0, u^i_1) e^{(i)} (0) \nonumber\\
     = & c (i, \Xi, u^i_0, u^i_1) \int_{\mathbb{R}^d} \varphi_{t, x} (0)
    \frac{1}{2} e^{2 W_{>}} (| v_1^{(i)} |^2 + | \nabla v_0^{(i)} |^2 + |
    v_0^{(i)} |^2 (\tmmathbf{C} (\Xi, i) - Z_{>}))  \label{ineq:ei}
  \end{align}
  and
  \begin{align}
    e (s)  \leqslant & c (\Xi, L, R, u^R_0, u^R_1, u^L_0, u^L_1) e (0)
    \nonumber\\
    = & c (\Xi, L, R, u^R_0, u^R_1, u^L_0, u^L_1) \int_{\mathbb{R}^d}
    \varphi_{t, x} (0) \frac{1}{2} e^{2 W_{>}} (| b_1 |^2 + | \nabla b_0 |^2 +
    | b_0 |^2 (\tmmathbf{C} (\Xi, R, L) - Z_{>})) 
  \end{align}
  hold for
  \[ 0 \leqslant s \leqslant t \text{ and } x \in \mathbb{R}^d
     \text{ s.t. }\mathfrak{C}_{(t, x)} \subset \left[ 0, \frac{L}{2} \right] \times B
     (L), \]
  where the backward light cone is defined as
  \begin{equation}
    \mathfrak{C}_{(t, x)} \assign \left\{ (s, y) \in \mathbb{R} \times
    \mathbb{R}^3 : 0 \leqslant s \leqslant t \text{ and }  y \in B (x, 2 (t
    - s)) \right\}
  \end{equation}
  as in Theorem \ref{thm:linear}.
  
  \begin{proof}
    In this case the difference of the solutions $u^{(R)}$ and $u^{(L)}$ with
    $R \geqslant L$ will not solve the same equation as in the linear case.
    Instead we make the observation that the difference $d \assign u^{(R)} -
    u^{(L)}$ in this case solves the equation
    \begin{align*}
      \partial^2_t d - H_L d = & \underset{= 0 \text{ in } B
      (L)}{\underbrace{(\chi_{B (R)} Z_{\leqslant} - \chi_{B (L)}
      Z_{\leqslant})}} u^{(R)} - d (| u^{(R)} |^2 + u^{(L)} u^{(R)} + |
      u^{(L)} |^2)\\
      (d, \partial_t d) |_{t = 0} = & (0, 0) .
    \end{align*}
    For future reference we also give the equation solved by $b \assign e^{-
    W_{>}} d$ and its ``weak'' formulation; in analogy to the previous
    sections we also write $u^{(i)} = e^{W_{>}} v^{(i)}$ for $i = R, L$
    \begin{align*}
      \partial^2_t b - e^{- W_{>}} H_L e^{W_{>}} b  = & \underset{= 0
      \text{ in } B (L)}{\underbrace{(\chi_{B (R)} Z_{\leqslant} - \chi_{B (L)}
      Z_{\leqslant})}} v^{(R)} - e^{2 W_{>}} b ( | v^{(R)} |^2 + v^{(L)}
      v^{(R)} + | v^{(L)} |^2)\\
      (d, \partial_t d) |_{t = 0} = & (0, 0) ;\\
        & \\
      (\phi, e^{2 W_{>}} \partial^2_t b - e^{W_{>}} H_L e^{W_{>}} b)  = &
      (\phi, (\chi_{B (R)} Z_{\leqslant} - \chi_{B (L)} Z_{\leqslant}) e^{2
      W_{>}} v^{(R)} - e^{4 W_{>}} b ( | v^{(R)} |^2 + v^{(L)} v^{(R)} + |
      v^{(L)} |^2))\\
      \text{for all } \phi & \in  \mathcal{H}^1 .
    \end{align*}
    We thus again compute integrating the gradient term as in the proof of
    Theorem \ref{thm:linear}
    \begin{align*}
      \frac{d}{d s} e (s)  = & \int_{\mathbb{R}^d} \frac{d}{d s} \varphi_{(t,
      x)} (s) \frac{1}{2} e^{2 W_{>}} (| \partial_t b (s) |^2 + | \nabla b (s)
      |^2 + (\tmmathbf{C} (\Xi, L) - Z_{>}) | b (s) |^2) +\\ &+ \int_{\mathbb{R}^d}
      \varphi_{(t, x)} (s) e^{2 W_{>}} \partial_t b (s) (\partial^2_t b (s) +
      (\tmmathbf{C} (\Xi, L) - Z_{>}) b (s))+ \int_{\mathbb{R}^d}
      \varphi_{(t, x)} (s) e^{2 W_{>}} \partial_t \nabla b (s) \nabla b (s)\\
       = & \int_{\mathbb{R}^d} \frac{d}{d s} \varphi_{(t, x)} (s)
      \frac{1}{2} e^{2 W_{>}} (| \partial_t b (s) |^2 + | \nabla b (s) |^2 +
      (\tmmathbf{C} (\Xi, L) - Z_{>}) | b (s) |^2) +\\ &+ \int_{\mathbb{R}^d}
      \varphi_{(t, x)} (s) e^{2 W_{>}} \partial_t b (s) (\partial^2_t b (s) -
      e^{- W_{>}} H_L e^{W_{>}} b (s) + e^{2 W_{>}} b (s) (| v^{(R)} |^2 +
      v^{(L)} v^{(R)} + | v^{(L)} |^2)) +\\ &+ \int_{\mathbb{R}^d} \varphi_{(t, x)}
      (s) e^{2 W_{>}} \partial_t b (s) b (s)  (\tmmathbf{C} (\Xi, L) + \chi_{B
      (L)} Z_{\leqslant}) +\\ &+ \int_{\mathbb{R}^d} \varphi_{(t, x)} (s) e^{4
      W_{>}} \partial_t b (s) b (s) (| v^{(R)} |^2 + v^{(R)} v^{(L)} + |
      v^{(L)} |^2) .
    \end{align*}
    Now we proceed similarly to the proof of Theorem \ref{thm:linear}. We
    need, however, one additional ingredient, namely bounds on $v^{(R)}$ and
    $v^{(L)}$ which allow us control the final term. Since they are all
    bounded analogoulsy we show how to bound one term and the others are
    analogous. We start by making a simple observation about our bump function
    $\varphi_{(t, x)}$, which is that
    \begin{align}
      \varphi_{(t, x)} (s) \leqslant \chi_{\tmop{supp} \varphi_{(t, x)} (s)} &
      = \chi_{B (x, 2 (t - s))} + \chi_{\tmop{supp} \varphi'_{(t, x)} (s) } 
      \label{ineq:bump}\\
      & = \chi_{B (x, 2 (t - s))} + \chi_{| \cdummy - x | - 2 (t - s) \in [0,
      1]} \\
      & = \chi_{B (x, 2 (t - s))} - \varphi'_{(t, x)} (s) \\
      & \text{and} \nonumber\\
      \chi_{B (x, 2 (t - s))} & \leqslant \varphi_{(t, x)} (s).
    \end{align}
    For $\varepsilon > 0$ we bound by Young, H{\"o}lder and the
    $\mathcal{H}^1 \hookrightarrow L^6$ Sobolev embedding
    \begin{align*}
      & \left| \int_{\mathbb{R}^3} \varphi_{(t, x)} (s) e^{4 W_{>}}
      \partial_t b (s) b (s) | v^{(R)} (s) |^2 \right|\leqslant\\
      \leqslant & C (\varepsilon) \int_{\mathbb{R}^3} \varphi_{(t, x)} (s)
      e^{2 W_{>}} | \partial_t b (s) |^2 + \varepsilon \int_{\mathbb{R}^3}
      \varphi_{(t, x)} (s) e^{6 W_{>}} | b (s) |^2 | v^{(R)} (s) |^4\\
      \leqslant & C (\varepsilon) \int \varphi_{(t, x)} (s) e^{2 W_{>}} |
      \partial_t b (s) |^2 + \varepsilon \left( \int_{\mathbb{R}^d}
      \varphi_{(t, x)} (s) | b (s) |^6 \right)^{\frac{1}{3}} \underset{0
      \leqslant \tau \leqslant t}{\sup} \left( \int_{\mathbb{R}^d}
      \varphi_{(t, x)} (\tau) | v^{(R)} (\tau) |^6 \right)^{\frac{2}{3}} \|
      e^{6 W_{>}} \|_{L^{\infty}}\\
      \overset{\eqref{ineq:bump}}{\leqslant} & C (\varepsilon) \int
      \varphi_{(t, x)} (s) e^{2 W_{>}} | \partial_t b (s) |^2 + \varepsilon
      \Bigg( \left( \int_{B (x, 2 (t - s))} | b (s) |^6 \right)^{\frac{1}{3}}
      +\\ &+ \left( \int_{\textrm{\tmop{supp}} \varphi'} | b (s) |^6
      \right)^{\frac{1}{3}} \Bigg) \underset{0 \leqslant \tau \leqslant
      t}{\sup} \left( \int_{\mathbb{R}^d} | v^{(R)} (\tau) |^6
      \right)^{\frac{2}{3}} \| e^{6 W_{>}} \|_{L^{\infty}}\\
      \leqslant & C (\varepsilon) \int \varphi_{(t, x)} (s) e^{2 W_{>}} |
      \partial_t b (s) |^2 + C \varepsilon \Big( \int_{B (x, 2 (t - s))} |
      \nabla b (s) |^2 + | b (s) |^2 +\\&+ \int_{\textrm{\tmop{supp}} \varphi'} |
      \nabla b (s) |^2 + | b (s) |^2 \Big) \underset{0 \leqslant \tau
      \leqslant t}{\sup} \left( \int_{\mathbb{R}^d} | \nabla v^{(R)} (\tau)
      |^2 + | v^{(R)} (\tau) |^2 \right)^2 \| e^{6 W_{>}} \|_{L^{\infty}}\\
      \leqslant & C (\varepsilon) \int \varphi_{(t, x)} (s) e^{2 W_{>}} |
      \partial_t b (s) |^2 + \varepsilon C (\Xi) \Big( \int_{\mathbb{R}^d}
      \varphi_{(t, x)} (s) e^{2 W_{>}} | \nabla b (s) |^2 -\\&-
      \int_{\mathbb{R}^d} \varphi'_{(t, x)} (s) e^{2 W_{>}} | \nabla b (s) |^2
      \Big) e^{\tilde{C} (\Xi, R) t} E_{\gg} (u^R_0, u^R_1)
    \end{align*}
    hence by choosing $\varepsilon$ small enough depending on the norms of the
    noise terms, the parameter $L$ and the $L_{[0, t]}^{\infty} \mathcal{H}^1 -$norm of
    $v^{(R)}, v^{(L)} $(which are bounded by the initial data by Theorem
    \ref{thm:trunc}) we get, after bounding the other terms in the same way
    \begin{align*}
      \frac{d}{d s} e (s) \leqslant & \int_{\mathbb{R}^d} \varphi'_{(t, x)}
      (s) e^{2 W_{>}} \Big( | \partial_t b (s) |^2 + | \nabla b (s) |^2 (1 -
      \varepsilon C (\Xi, L, R, u^R_0, u^R_1, u^L_0, u^L_1, \varepsilon)) +\\&+
      \left( \tmmathbf{C} (\Xi, L, R) - \varepsilon C (\Xi, L, R, u^R_0,
      u^R_1, u^L_0, u^L_1, \varepsilon) \right) | b (s) |^2  \Big) + \bar{C} (\Xi, L, R, u^R_0, u^R_1, u^L_0, u^L_1,
      \varepsilon) e (s) \\
       \leqslant & c (\Xi, L, R, u^R_0, u^R_1, u^L_0, u^L_1) e (s)
    \end{align*}
    using the negativity of the first term as in the proof of Theorem
    \ref{thm:linear} and choosing a suitably small $\varepsilon > 0$.
    
    Proving the bound \eqref{ineq:ei} is analogous.
    
    This finishes the proof.
  \end{proof}
\end{theorem}

This leads us to the final result, which tells us that the PDE
\eqref{eqn:cubicformal} whose weak formulation is \eqref{eqn:weakcubic} is
globally well-posed in space and time.

\begin{theorem}[Global well-posedness of the cubic multiplicative stochastic
wave equation]
  Let $T > 0$ and initial data $(u_0, u_1) \in e^{W_{>}} \mathcal{H}^1 \times
  L^2$. Then there exists a unique solution
  \[ u \in C ([0, T] ; e^{W_{>}} \mathcal{H}_{\tmop{loc}}^1) \cap C^1 ([0, T]
     ; L_{\tmop{loc}}^2) \]
  to the PDE \eqref{eqn:weakcubic} with continous dependence of the localised
  norms on the data inside the backward light cones.
  
  \begin{proof}
    By Theorem \ref{thm:trunc} we have local space-time existence for
    solutions to \eqref{eqn:weakcubic} and then Theorem \ref{thm:finitecubic}
    gives us local uniqueness and local continuous dependence on the data in
    the backward light cone.
  \end{proof}
\end{theorem}

\appendix\section{Appendix: Some results on Besov spaces and weights}

\begin{definition}[Weighted Besov spaces, {\cite{GHglobal}}]
  \label{def:besov}For $\nu \in \mathbb{R}$ we consider the following class of
  weights
  \[ \langle x \rangle^{\nu} = (1 + | x |^2)^{\frac{\nu}{2}} \quad x \in
     \mathbb{R}^d \]
  and define the weighted $L^p$ space w.r.t. this weight as
  \[ L_{\langle \cdummy \rangle^{\nu}}^p \assign \{ f \in \mathcal{S}'
     (\mathbb{R}^d) : f \cdummy \langle \cdummy \rangle^{\nu} \in L^p \}, p
     \in [1, \infty], \]
  whose norm is defined as
  \[ \| f \|_{L_{\langle \cdummy \rangle^{\nu}}^p} \assign \| f \langle
     \cdummy \rangle^{\nu} \|_{L^p} . \]
  Moreover, for a Littlewood-Paley decomposition $(\Delta_i)_{i \geqslant -
  1},$ one defines weighted Besov spaces as
  \[ B_{p, q, \nu}^s \assign \{ f \in \mathcal{S}' (\mathbb{R}^d) : \| f
     \|_{B_{p, q, \nu}^s} < \infty  \}, \text{with the norm } \| f \|_{B_{p,
     q, \nu}^s} \assign \left\| \| \Delta_i f \|_{L_{\langle \cdummy
     \rangle^{\nu}}^p} \right\|_{\ell_i^q} . \]
  In particular, for $\nu = 0$ this agrees with the usual unweighted Besov
  space, moreover we write
  \[ \mathcal{C}_{\langle \cdummy \rangle^{\nu}}^s \assign B_{\infty, \infty,
     \nu}^s \]
  and refer to it as a weighted Besov-H{\"o}lder space.
  
  We note that one can analogously define these spaces on the torus
  $\mathbb{T}^d$, however one usually does not need weights in that setting
  due to the compactness of $\mathbb{T}^d .$
\end{definition}

\begin{lemma}[Besov regularity of indicator functions, Theorem 2 in
{\cite{charreg}}]
  \label{lem:charreg}Let $A \subset \mathbb{R}^d $ be a bounded set with
  finite perimeter, then we have
  \[ \chi_A \in B_{p, \infty}^{\frac{1}{p}}  \text{ for any } p \in [1, \infty)
     . \]
\end{lemma}

\begin{lemma}[Fractional Leibnitz for Besov spaces, Proposition A.7 in
{\cite{mourrat3d}}]
  \label{lem:leib}
  
  For $s > 0$ and $1 \leqslant p, q, p_1, p_2, p_3, p_4 \leqslant \infty$ s.t.
  \[ \frac{1}{p} = \frac{1}{p_1} + \frac{1}{p_2} = \frac{1}{p_3} +
     \frac{1}{p_4}, \]
  we have the bound
  \[ \| f \cdummy g \|_{B_{p, q}^s} \lesssim \| f \|_{B_{p_1 q}^s} \| g
     \|_{L^{p_2}} + \| f \|_{L^{p_3}} \| g \|_{B_{p_4, q}^s} \]
\end{lemma}

\begin{lemma}[Besov embedding and interpolation]
  \label{lem:embedding}Let $1 \leqslant p_1 \leqslant p_2 \leqslant \infty$,
  $1 \leqslant q_1 \leqslant q_2 \leqslant \infty$ and $\alpha \in
  \mathbb{R}$. Then we have the continuous embeddings
  \[ B_{p_1, q_1}^{\alpha} \hookrightarrow B_{p_2, q_2}^{\alpha - d \left(
     \frac{1}{p_1} - \frac{1}{p_2} \right)} \]
  as well as
  \begin{eqnarray*}
    B_{p, 1}^{d \left( \frac{1}{p} - \frac{1}{q} \right)} \hookrightarrow L^q
    &  \text{for } 1 \leqslant p \leqslant q \leqslant \infty\\
    & \text{and}\\
    B_{p, 2}^0 \hookrightarrow L^p & \text{for } 2 \leqslant p < \infty .
  \end{eqnarray*}
  Moreover, for $\alpha_1 < \alpha_2$ and $\theta \in (0, 1)$ we have for any
  $p, r \in [1, \infty]$ the bounds
  \begin{eqnarray*}
    \| u \|_{B^{\theta s_1 + (1 - \theta) s_2}_{p, r}} \lesssim & \| u
    \|^{\theta}_{B^{s_1}_{p, r}} \| u \|^{1 - \theta}_{B^{s_2}_{p, r}}\\
    \text{and} & \\
    \| u \|_{B^{\theta s_1 + (1 - \theta) s_2}_{p, 1}} \lesssim & \| u
    \|^{\theta}_{B^{s_1}_{p, \infty}} \| u \|^{1 - \theta}_{B^{s_2}_{p,
    \infty}} .
  \end{eqnarray*}
  \begin{proof}
    See Propositions 2.18, 2.22 and 2.39 as well as Theorem 2.40 in
    {\cite{BCD}}.
  \end{proof}
\end{lemma}

\begin{proposition}[{\cite{rychkov}},Extension and restriction operators on
Besov spaces]
  \label{prop:ext}Let $\Omega \subset \mathbb{R}^d$ be an open set with
  Lipschitz boundary, then the Besov space $B^s_{p, q} (\Omega)$ is defined as
  \begin{equation}
    B^s_{p, q} (\Omega) \assign \{ f \in \mathcal{D}' (\Omega) : \| f
    \|_{B^s_{p, q} (\Omega)} \assign \inf \{ \| g \|_{B^s_{p, q}
    (\mathbb{R}^d) } g \in \mathcal{D}' (\mathbb{R}^d) : g|_{\Omega} = f \} \}
    .
  \end{equation}
  Then there exists a bounded extension operator $E_{\Omega}$ s.t. for $f \in
  B^s_{p, q} (\Omega)$
  \[ \| E_{\Omega} f \|_{B^s_{p, q} (\mathbb{R}^d)} \lesssim \| f \|_{B^s_{p,
     q} (\Omega)} . \]
  Moreover, by the definition of the space, one has the bound for the
  restriction to $\Omega$
  \[ \| g|_{\Omega} \|_{B^s_{p, q} (\Omega)} \lesssim \| g \|_{B^s_{p, q}
     (\mathbb{R}^d)} \]
  for any $g \in B^s_{p, q} (\mathbb{R}^d) .$
\end{proposition}

\begin{remark}
  By Proposition \ref{prop:ext} the results from Lemma \ref{lem:embedding} are
  still true on the space $B^s_{p, q} (\Omega) .$
\end{remark}

\bigskip
\noindent \textcolor{gray}{$\bullet$} I. Zachhuber --  , FU Berlin, D-14195 Berlin, Germany\\
{\it E-mail}: immanuel.zachhuber@fu-berlin.de

\begin{thebibliography}{10}
  \bibitem[1]{allez_continuous_2015}Romain Allez  and  Khalil Chouk.
  {\newblock}The continuous Anderson Hamiltonian in dimension two.
  {\newblock}\tmtextit{ArXiv:1511.02718 [math]}, nov 2015. {\newblock}ArXiv:
  1511.02718.{\newblock}
  
  \bibitem[2]{BCD} Hajer Bahouri, Jean-Yves Chemin, and Rapha{\"e}l Danchin.
  \newblock\textit{Fourier analysis and nonlinear partial differential
  equations},  volume  343. {\newblock}Springer Science \& Business Media,
  2011.{\newblock}
  
  \bibitem[3]{BRfinite}Viorel Barbu  and  Michael R{\"o}ckner. {\newblock}The
  finite speed of propagation for solutions to nonlinear stochastic wave
  equations driven by multiplicative noise. {\newblock}\textit{J.
  Differential Equations}, 255(3):560--571, 2013.{\newblock}
  
  \bibitem[4]{CvZ}Khalil Chouk  and  Willem van~Zuijlen.
  {\newblock}Asymptotics of the eigenvalues of the Anderson Hamiltonian with
  white noise potential in two dimensions. {\newblock}\tmtextit{The Annals of
  Probability}, 49(4):1917--1964, 2021.{\newblock}
  
  \bibitem[5]{debglobal}Arnaud Debussche  and  J{\"o}rg Martin.
  {\newblock}Solution to the stochastic Schr{\"o}dinger equation on the full
  space. {\newblock}\tmtextit{Nonlinearity}, 32(4):1147--1174,
  2019.{\newblock}
  
  \bibitem[6]{debussche2016schr}Arnaud Debussche  and  Hendrik Weber.
  {\newblock}The Schr{\"o}dinger equation with spatial white noise potential.
  {\newblock}\textit{ArXiv preprint arXiv:1612.02230}, 2016.{\newblock}
  
  \bibitem[7]{Evans10}Lawrence C. Evans. {\newblock}{Partial
  Differential Equations},  volume~19  of {Graduate Studies in
  Mathematics}. {\newblock}American Mathematical Society, Providence, RI,
  Second  edition, 2010.{\newblock}
  
  \bibitem[8]{GKO2}M.~Gubinelli, H.~Koch, and  T.~Oh.
  {\newblock}Renormalization of the two-dimensional stochastic nonlinear wave
  equation. {\newblock}\tmtextit{Transactions of the American Mathematical Society 370.10 (2018): 7335-7359.}
  
  \bibitem[9]{GUZ}M.~Gubinelli, B.~Ugurcan, and  I.~Zachhuber.
  {\newblock}Semilinear evolution equations for the Anderson Hamiltonian in
  two and three dimensions. {\newblock}\tmtextit{Stoch. Partial Differ. Equ.
  Anal. Comput.}, 8(1):82--149, 2020.{\newblock}
  
  \bibitem[10]{GHglobal}Massimiliano Gubinelli  and  Martina Hofmanov{\'a}.
  {\newblock}Global solutions to elliptic and parabolic $\Phi^4$ models in
  Euclidean space. {\newblock}\tmtextit{Comm. Math. Phys.}, 368(3):1201--1266,
  2019.{\newblock}
  
  \bibitem[11]{GIP}Massimiliano Gubinelli, Peter Imkeller, and  Nicolas
  Perkowski. {\newblock}Paracontrolled distributions and singular PDEs.
  {\newblock}In \tmtextit{Forum of Mathematics, Pi},  volume~3. Cambridge
  University Press, 2015.{\newblock}
  
  \bibitem[12]{hairer_theory_2014}M.~Hairer. {\newblock}A theory of regularity
  structures. {\newblock}\tmtextit{Inventiones mathematicae}, 198(2):269--504,
  mar 2014.{\newblock}
  
  \bibitem[13]{HairerLabbe15}Martin Hairer  and  Cyril Labb{\'e}. {\newblock}A
  simple construction of the continuum parabolic Anderson model on
  $\tmmathbf{R}^2$. {\newblock}\textit{Electron. Commun. Probab.}, 20:0,
  2015.{\newblock}
  
  \bibitem[14]{hairerlabbe3d}Martin Hairer  and  Cyril Labb{\'e}.
  {\newblock}Multiplicative stochastic heat equations on the whole space.
  {\newblock}\tmtextit{Journal of the European Mathematical Society},
  20(4):1005--1054, 2018.{\newblock}
  
  
    \bibitem[15]{jape}Jagannath, Aukosh, and Nicolas Perkowski.{\newblock} "A simple construction of the dynamical $\Phi^ 4_3$ model." \textit{arXiv preprint arXiv:2108.13335 (2021).}
    {\newblock}
  
  
  \bibitem[16]{labbe}Cyril Labb{\'e}. {\newblock}The continuous Anderson
  Hamiltonian in $d \leq 3$. {\newblock}\tmtextit{J. Funct. Anal.},
  277(9):3187--3235, 2019.{\newblock}
  
  \bibitem[17]{LHfinite}Fei Liang  and  Zhe Hu. {\newblock}The finite speed of
  propagation for solutions to stochastic viscoelastic wave equation.
  {\newblock}\tmtextit{Bound. Value Probl.}, 2019.{\newblock}
  
  \bibitem[18]{mourrat3d}Jean-Christophe Mourrat  and  Hendrik Weber.
  {\newblock}The dynamic $\Phi^4_3$ model comes down from infinity.
  {\newblock}\tmtextit{Comm. Math. Phys.}, 356(3):673--753, 2017.{\newblock}
  
  \bibitem[19]{mouzardstrichartz}Antoine Mouzard  and  Immanuel Zachhuber.
  {\newblock}Strichartz inequalities with white noise potential on compact
  surfaces. {\newblock}\textit{arXiv preprint arXiv:2104.07940}, 2021.{\newblock}
  
  \bibitem[20]{reedsimon1}Michael Reed  and  Barry Simon.
  {\newblock}\tmtextit{Methods of modern mathematical physics. I. Functional
  Analysis}. {\newblock}Academic Press, New York-London, 1972.{\newblock}
  
  \bibitem[21]{rychkov}Vyacheslav~S.~Rychkov. {\newblock}On restrictions and
  extensions of the Besov and Triebel-Lizorkin spaces with respect to
  Lipschitz domains. {\newblock}\tmtextit{J. London Math. Soc. (2)},
  60(1):237--257, 1999.{\newblock}
  
  \bibitem[22]{charreg}Winfried Sickel. {\newblock}On the regularity of
  characteristic functions. {\newblock}In \textit{Anomalies in Partial
  Differential Equations},  pages  395--441. Springer, 2021.{\newblock}
  
  \bibitem[23]{tartar}Luc Tartar. {\newblock}\tmtextit{Topics in nonlinear
  analysis},  volume~13  of \tmtextit{Publications Math{\'e}matiques d'Orsay
  78}. {\newblock}Universit{\'e} de Paris-Sud, D{\'e}partement de
  Math{\'e}matique, Orsay, 1978.{\newblock}
  
  \bibitem[24]{tolomeoglobal}Leonardo Tolomeo. {\newblock}Global
  well-posedness of the two-dimensional stochastic nonlinear wave equation on
  an unbounded domain. {\newblock}\tmtextit{ArXiv preprint arXiv:1912.08667},
  2019.{\newblock}
  
  \bibitem[25]{zachhuber2019strichartz}Immanuel Zachhuber.
  {\newblock}Strichartz estimates and low-regularity solutions to
  multiplicative stochastic nls. {\newblock}\tmtextit{ArXiv preprint
  arXiv:1911.01982}, 2019.{\newblock}
\end{thebibliography}
\end{document}